\documentclass[12pt]{article}
\usepackage[english]{babel}
\usepackage{amsmath,amsthm,amssymb}
\usepackage{amsfonts}
\usepackage{enumerate}
\usepackage{hyperref}
\hypersetup{colorlinks,
             linkcolor=black, anchorcolor=black,
             citecolor=black, urlcolor=black}
\newtheorem{thm}[equation]{Theorem}
\newtheorem{cor}[equation]{Corollary}
\newtheorem{lem}[equation]{Lemma}
\newtheorem{prop}[equation]{Proposition}

\newtheorem{mainthm}{Theorem}

\theoremstyle{definition}
\newtheorem{defn}[equation]{Definition}
\newtheorem{remark}[equation]{Remark}
\newtheorem{observation}[equation]{Observation}
\newtheorem{question}[equation]{Question}
\newtheorem{example}[equation]{Example}
\newtheorem*{notation}{Notation}

\numberwithin{equation}{section}
\begin{document}

\title{\bf\Large Extension of Fujimoto's uniqueness theorems}%
\author{Kai Zhou}%
\date{}%
\maketitle
\def\thefootnote{}

\begin{abstract}
  Hirotaka Fujimoto considered two meromorphic maps $ f $ and $ g $ of $\mathbb{C}^m $ into $\mathbb{P}^n $ such that $ f^*(H_j)=g^*(H_j)$ ($ 1\leq j\leq q $) for $ q $ hyperplanes $ H_j $ in $\mathbb{P}^n $ in general position and proved $ f=g $ under suitable conditions. This paper considers the case where $ f $ is into $\mathbb{P}^n $ and $ g $ is into $\mathbb{P}^N $ and gives extensions of some of Fujimoto's uniqueness theorems. The dimensions $ N $ and $ n $ are proved to be equal under suitable conditions. New and interesting phenomena also occur. 
\end{abstract}

{\small \tableofcontents}
\hypersetup{linkcolor=blue, anchorcolor=blue,
             citecolor=blue, urlcolor=blue}

\section{Introduction and main results}      \label{sec:Introduction_main_results}

In the 1920s, George P\'olya \cite{Polya1921} and Rolf Nevanlinna \cite{Nevanlinna1926} proved some uniqueness theorems for meromorphic functions. For instance, Nevanlinna \cite{Nevanlinna1926} obtained the famous five-value theorem which says that two nonconstant meromorphic functions on the complex plane $\mathbb{C}$ are necessarily equal if they have the same inverse image sets for five distinct points in the Riemann sphere.

In the 1970s, Hirotaka Fujimoto \cite{Fujimoto75,Fujimoto76,Fujimoto78} gave some generalizations of the uniqueness results of P\'olya \cite{Polya1921} and Nevanlinna \cite{Nevanlinna1926} to the case of meromorphic maps into the $ n$-dimensional complex projective space $\mathbb{P}^n.$ For instance, Fujimoto proved in \cite{Fujimoto75} the following famous uniqueness theorem.

\begin{thm}     \label{thm:(3n+2)thm_of_Fujimoto}
Let $ f $ and $ g $ be meromorphic maps of $\mathbb{C}^m $ into $\mathbb{P}^n $ such that $ f(\mathbb{C}^m)\not\subseteq H_j $, $ g(\mathbb{C}^m)\not\subseteq H_j $ and $ f^*(H_j)=g^*(H_j)$ $(1\leq j\leq q)$ for $ q $ hyperplanes $ H_1,\dots, H_q $ in $\mathbb{P}^n $ in general position. If $ g $ is linearly non-degenerate and $ q\geq 3n+2,$ then $ f=g.$
\end{thm}

Since Fujimoto's research, there have been lots of papers which deal with the uniqueness problem in higher dimensions, see for example \cite{Drouilhet81,Smiley83,Fujimoto99,Aihara00,ChenYan09,Quang11} and the references therein. In the uniqueness results of these papers, the target spaces of different holomorphic maps are same. However the uniqueness problem where the identity of target spaces is not a prior assumption is seldom studied. Katsutoshi Yamanoi considered in \cite[subsection 3.2]{Yamanoi04_2} the uniqueness problem for holomorphic curves into possibly different abelian varieties. In \cite{CorvajaNoguchi12}, Pietro Corvaja and Junjiro Noguchi extended Yamanoi's result to the case of semi-abelian varieties. A corollary of the main theorem of \cite{CorvajaNoguchi12} is the following.

\begin{thm}[see Corollary 1.14 in \cite{CorvajaNoguchi12}]       \label{thm:two_holomorphic_curves_in_C*n_C*N_sharing_a_hyperplane}
Let $ f:\mathbb{C}\to(\mathbb{C}^*)^n $ and $ g:\mathbb{C}\to(\mathbb{C}^*)^N $ be algebraically non-degenerate holomorphic maps such that $ f^{-1}(D_n)=g^{-1}(D_N).$ Then $ N=n $ and there is an isomorphism $\varphi:(\mathbb{C}^*)^n\to(\mathbb{C}^*)^n $ such that $ g=\varphi\circ f $ and $\varphi(D_n)=D_N.$
\end{thm}
Here, for an integer $ n\geq 1 $, $ D_n $ denotes the following hyperplane in $(\mathbb{C}^*)^n $:
\[
   D_n=\{(x_1,\dots,x_n)\in(\mathbb{C}^*)^n\,|\, x_1+\dots+x_n=1\}.
\]

The aim of this paper is to extend the uniqueness theorems of Fujimoto to the case where the target spaces are possibly different dimensional complex projective spaces.

One of our main results is the following which extends Fujimoto's Theorem \ref{thm:(3n+2)thm_of_Fujimoto}.

\begin{mainthm}          \label{mainthm:q>=N+2n+2}
Let $ N\geq n\geq 1 $ be integers. Let $\{H_j\}_{j=1}^q $ be hyperplanes in $\mathbb{P}^n $ in general position and $\{H'_j\}_{j=1}^q $ be hyperplanes in $\mathbb{P}^N $ in general position. Let $ f:\mathbb{C}^m\to\mathbb{P}^n $ and $ g:\mathbb{C}^m\to\mathbb{P}^N $ be meromorphic maps such that $ f(\mathbb{C}^m)\not\subseteq H_j $, $ g(\mathbb{C}^m)\not\subseteq H'_j $ and $ f^*(H_j)=g^*(H'_j)$ for any $ 1\leq j\leq q.$ If $ g $ is linearly non-degenerate and $ q\geq N+2n+2,$ then $ N=n $ and there are a projective linear transformation $ L $ of $\mathbb{P}^n $ and $(n+2)$ distinct indices $ k_1,\dots, k_{n+2}\in\{1,\dots,q\}$ such that $ f=L(g)$ and
 \[
    L(H'_{k_i})=H_{k_i}, \quad 1\leq i\leq n+2.
 \]
\end{mainthm}

As a matter of fact, we shall prove in Section \ref{sec:Extension_of_(3n+2)thm} an improved version of Theorem \ref{mainthm:q>=N+2n+2}, namely Theorem \ref{thm:q>=min(N+2n+2,max())}.

In \cite{Fujimoto75}, Fujimoto proved also the following theorem.
\begin{thm}     \label{thm:(3n+1)thm_of_Fujimoto}
Let $ f $, $ g $ and $ H_1,\dots, H_q $ be as in Theorem {\rm \ref{thm:(3n+2)thm_of_Fujimoto}}. If $ n\geq 2,$ and $ g(\mathbb{C}^m)$ is not contained in any hypersurface of degree $\leq 2 $ in $\mathbb{P}^n,$ and $ q\geq 3n+1,$ then $ f=g.$
\end{thm}

We shall extend the above theorem of Fujimoto to the following.

\begin{mainthm}       \label{mainthm:q>=max(3n+1,N+1.5n+1.5)}
Let $ N\geq n\geq 1 $ be integers. Let $\{H_j\}_{j=1}^q $, $\{H'_j\}_{j=1}^q $, $ f $ and $ g $ be as in Theorem {\rm \ref{mainthm:q>=N+2n+2}}. If $ g(\mathbb{C}^m)$ is not contained in any hypersurface of degree $\leq\min\{n+1,N-n+2\}$ in $\mathbb{P}^N $ and
 \[
    q\geq \max\big\{3n+1, N+\frac{3}{2}n+\frac{3}{2}\big\},
 \]
then $ N=n $ and there is a projective linear transformation $ L $ of $\mathbb{P}^n $ such that $ f=L(g)$ and $ L $ satisfies the following condition: either
 \[
    L(H'_{k_i})=H_{k_i}, \quad 1\leq i\leq n+2,
 \]
for $(n+2)$ distinct indices $ k_1,\dots, k_{n+2}\in\{1,\dots,q\},$ or
 \[
    L(H'_{k_i})=H_{k_i}, \,\, 1\leq i\leq n+1, \quad\mbox{and}\quad L(H'_{k_{n+2}})=H_{k_{n+3}},\,\, L(H'_{k_{n+3}})=H_{k_{n+2}},
 \]
for $(n+3)$ distinct indices $ k_1,\dots, k_{n+3}\in\{1,\dots,q\}.$
\end{mainthm}

The above Theorem \ref{mainthm:q>=max(3n+1,N+1.5n+1.5)} will be proved in Section \ref{sec:Extension_of_(3n+1)thm}. Indeed, we shall prove Theorem \ref{thm:q>=max(3n+1,min(2N+1,N+1.5n+1.5))} which slightly improves Theorem \ref{mainthm:q>=max(3n+1,N+1.5n+1.5)}.

\begin{remark}        \label{remk:no_L_fixes_(n+1)Hi_and_exchanges_another_two_when_n>=2}
For any $(n+3)$ hyperplanes $ H_1,\dots, H_{n+3}$ in $\mathbb{P}^n $ \emph{in general position} where $ n\geq 2,$ one can easily prove that there \emph{does not exist} a projective linear transformation $ L $ of $\mathbb{P}^n $ such that
 \[
    L(H_i)=H_i, \,\, 1\leq i\leq n+1, \quad\mbox{and}\quad L(H_{n+2})=H_{n+3},\,\, L(H_{n+3})=H_{n+2}.
 \]
So Theorem \ref{mainthm:q>=max(3n+1,N+1.5n+1.5)} recovers Fujimoto's Theorem \ref{thm:(3n+1)thm_of_Fujimoto}.
\end{remark}

\begin{remark}        \label{remk:f=L(g)_sharing_any_q_hyperplanes}
Let $ g:\mathbb{C}^m\to\mathbb{P}^n $ be any linearly non-degenerate meromorphic map. Let $\{H'_j\}_{j=1}^q $ be any hyperplanes in $\mathbb{P}^n $ in general position. Let $ L $ be a projective linear transformation of $\mathbb{P}^n.$ Then $ f:=L(g)$ is a linearly non-degenerate meromorphic map and $ H_j:=L(H'_j)$, $ 1\leq j\leq q,$ are hyperplanes in general position. The meromorphic maps $ f $ and $ g $ satisfy the condition that
 \[
    f^*(H_j)=g^*(H'_j), \quad 1\leq j\leq q.
 \]
\end{remark}

The following is a direct corollary of Theorems \ref{thm:q>=min(N+2n+2,max())} and \ref{thm:q>=max(3n+1,min(2N+1,N+1.5n+1.5))}.
\begin{cor}         \label{cor:(3n+2)_(3n+1)_when_n<=N<=1.5n}
Let $ N $ and $ n $ be integers such that $ 1\leq n\leq N\leq \frac{3}{2}n.$ Let $\{H_j\}_{j=1}^q $, $\{H'_j\}_{j=1}^q $, $ f $ and $ g $ be as in Theorem {\rm \ref{mainthm:q>=N+2n+2}}.
 \begin{enumerate}[\rm (i)]
   \item If $ g $ is linearly non-degenerate and $ q\geq 3n+2,$ then the conclusion of Theorem {\rm \ref{mainthm:q>=N+2n+2}} holds.
   \item If $ g(\mathbb{C}^m)$ is not contained in any hypersurface of degree $\leq (N-n+2)$ in $\mathbb{P}^N $ and $ q\geq 3n+1,$ then the conclusion of Theorem {\rm \ref{mainthm:q>=max(3n+1,N+1.5n+1.5)}} holds.
 \end{enumerate}
\end{cor}

Moreover, Fujimoto proved in \cite{Fujimoto76} and \cite{Fujimoto78} the following deep theorem.
\begin{thm}     \label{thm:(2n+3)thm_of_Fujimoto}
Let $ f $ and $ g $ be meromorphic maps of $\mathbb{C}^m $ into $\mathbb{P}^n $ such that $ f(\mathbb{C}^m)\not\subseteq H_j $, $ g(\mathbb{C}^m)\not\subseteq H_j $ and $ f^*(H_j)=g^*(H_j)$ $(1\leq j\leq 2n+3)$ for $(2n+3)$ hyperplanes $ H_j $ in $\mathbb{P}^n $ in general position. If $ g $ is algebraically non-degenerate, then $ f=g.$
\end{thm}

We refer the reader to \cite{zk23_ANote} for a note on Fujimoto's proof of the above theorem.

Fujimoto obtained the above Theorem \ref{thm:(2n+3)thm_of_Fujimoto} after a very deep analysis of the case when the number of hyperplanes is $(2n+2)$ and $ g $ is algebraically non-degenerate.

When trying to extend Theorem \ref{thm:(2n+3)thm_of_Fujimoto}, we get a partial result, namely, Theorem \ref{thm:N=n_when_g_algebraically_nondegenerate} which is also one of the main results of this paper. Let $ N $, $ n $, $\{H_j\}_{j=1}^q $, $\{H'_j\}_{j=1}^q $, $ f $ and $ g $ be as in Theorem \ref{mainthm:q>=N+2n+2}. Theorem \ref{thm:N=n_when_g_algebraically_nondegenerate} says that the dimensions $ N $ and $ n $ are necessarily equal if $ g $ is algebraically non-degenerate and $ q\geq N+n+2.$ So the extension of Theorem \ref{thm:(2n+3)thm_of_Fujimoto} becomes studying the case when there are two families $\{H_j\}_{j=1}^{2n+3}$ and $\{H'_j\}_{j=1}^{2n+3}$ of hyperplanes in a complex projective space $\mathbb{P}^n.$ We cannot yet get a satisfactory result and a discussion about this will be given in Section \ref{sec:Open_questions_further_remarks}.

This paper is organized as follows. In Section \ref{sec:Preliminaries}, we recall some definitions and include some auxiliary results. In particular, we shall give some combinatorial lemmas which are essentially due to Fujimoto. In Section \ref{sec:Extension_of_(3n+2)thm}, we consider the extension of Theorem \ref{thm:(3n+2)thm_of_Fujimoto} and two theorems shall be proved. In Section \ref{sec:Extension_of_(3n+1)thm}, we give an extension of Theorem \ref{thm:(3n+1)thm_of_Fujimoto}. In Section \ref{sec:Zariski_closure_Vfg}, we consider the dimension of the Zariski closure $ V_{f\times g}$ of the image of $ f\times g $ in $\mathbb{P}^n\times\mathbb{P}^N.$ The main theorem of Section \ref{sec:Zariski_closure_Vfg} will be essentially used in Section \ref{sec:N=n_when_g_algebraically_nondegenerate}. We shall also show that $ V_{f\times g}$ is an irreducible algebraic set of dimension $\leq N.$ In Section \ref{sec:N=n_when_g_algebraically_nondegenerate}, we shall show the equality of dimensions when $ g $ is algebraically non-degenerate and $ q\geq N+n+2.$ In Section \ref{sec:Open_questions_further_remarks}, we shall give further remarks about the extension of Theorem \ref{thm:(2n+3)thm_of_Fujimoto} and also ask a question concerning the main theorem of Section \ref{sec:N=n_when_g_algebraically_nondegenerate}. After Section \ref{sec:Open_questions_further_remarks}, there are two appendices in which we give proofs of some lemmas or conclusions in previous sections.

\paragraph{Acknowledgements.}  The author thanks Professor Lu Jin and Professor Qiming Yan for their constant encouragements. The author also thanks for Professor Qiming Yan's constant help.

\section{Preliminaries}       \label{sec:Preliminaries}

\paragraph{2.1\, Meromorphic maps and hyperplanes in complex projective spaces.}

Let $ f $ be a meromorphic map of $\mathbb{C}^m $ into $\mathbb{P}^n.$ Denote by $ I(f)$ the indeterminacy locus of $ f,$ which is an analytic subset of $\mathbb{C}^m $ of codimension $\geq 2.$ A \emph{reduced representation} of $ f $ is a $(n+1)$-tuple $(f_0,\dots, f_n)$ of holomorphic functions on $\mathbb{C}^m $ such that
\[
   \{f_0=f_1=\dots=f_n=0\}= I(f)
\]
and
\[
   f(z)=[f_0(z):\dots:f_n(z)] \quad \forall z\in\mathbb{C}^m\setminus I(f).
\]
Any meromorphic map of $\mathbb{C}^m $ into $\mathbb{P}^n $ has a reduced representation.

We also recall the following definition.
\begin{defn}        \label{defn:linear_and_algebraic_nondegeneracy}
A meromorphic map of $\mathbb{C}^m $ into $\mathbb{P}^n $ is said to be \emph{linearly} (resp. \emph{algebraically}) \emph{non-degenerate}, if its image is not contained in any hyperplane (resp. hypersurface) in $\mathbb{P}^n.$
\end{defn}

Let
\[
   H=\big\{[x_0:\dots:x_n]\in\mathbb{P}^n\,|\, a_0x_0+\dots+a_nx_n=0\big\}
\]
be a hyperplane in $\mathbb{P}^n $ where $(a_0,\dots,a_n)\in\mathbb{C}^{n+1}\setminus\{(0,\dots,0)\}.$ We shall say that the linear form $ a_0X_0+\dots+a_nX_n $ defines the hyperplane $ H.$

\begin{defn}    \label{defn:hyperplanes_in_general_position}
Let $\{H_j\}_{j=1}^q $ be a family of hyperplanes in $\mathbb{P}^n.$ Take a linear form $ L_j=a^j_0 X_0+\dots+a^j_n X_n $ that defines $ H_j $ for each $ 1\leq j\leq q.$ The family $\{H_j\}_{j=1}^q $ of hyperplanes is said to be \emph{in general position}, if, for any distinct indices $ i_1,\dots, i_k\in \{1,\dots, q\}$ with $ 1\leq k\leq n+1,$ the linear forms $ L_{i_1},\dots, L_{i_k}$ are linearly independent over $\mathbb{C}.$
\end{defn}

Let $\{H_j\}_{j=1}^q $ and $\{L_j\}_{j=1}^q $ be as above. If $ q\geq n+1,$ then $\{H_j\}_{j=1}^q $ is in general position if and only if
\[
   \det\big(a^{i_k}_0,\dots,a^{i_k}_n;\, 1\leq k\leq n+1\big)\neq 0
\]
for any $(n+1)$ distinct indices $ i_1,\dots, i_{n+1}\in\{1,\dots,q\}.$

\begin{defn}      \label{defn:pullback_divisor}
Let $ f:\mathbb{C}^m\to\mathbb{P}^n $ be a meromorphic map and let $ H $ be a hyperplane in $\mathbb{P}^n.$ By taking a reduced representation $(f_0,\dots,f_n)$ of $ f $ and taking a linear form $ a_0 X_0+\dots+a_n X_n $ that defines $ H,$ we define
 \[
    (f,H):=a_0f_0+\dots+a_nf_n,
 \]
which is a holomorphic function on $\mathbb{C}^m.$ When $ f(\mathbb{C}^m)\not\subseteq H,$ the pullback divisor $ f^*H $ is defined to be the zero divisor of the holomorphic function $(f,H),$ which is independent of the different choices of reduced representations of $ f $ and linear forms that define $ H.$
\end{defn}

\paragraph{2.2\, The Borel Lemma.}

Following Fujimoto's method, we shall essentially use the following result which is referred to as the Borel Lemma.

\begin{thm}[Borel Lemma, see \cite{Borel1897} and Corollary 4.2 in \cite{Fujimoto75}]       \label{thm:Borel_Lemma}
Let $ h_1,\dots, h_q $ be nowhere zero holomorphic functions on $\mathbb{C}^m $ which satisfy the following equation:
 \[
    c_1h_1+\dots+c_qh_q \equiv 0,
 \]
where $ c_1,\dots, c_q $ are nonzero complex numbers. Then, for each $ h_i,$ there is some $ h_j $ with $ j\neq i $ such that $ h_j/h_i $ is constant.
\end{thm}

We now give some notations.

\begin{notation}
Denote by $\mathcal{H}^*=\mathcal{H}^*_m $ the multiplicative group of all nowhere zero holomorphic functions on $\mathbb{C}^m.$ We regard the set $\mathbb{C}^*$ of all nonzero complex numbers as a subgroup of $\mathcal{H}^*.$ For any $ h\in\mathcal{H}^*,$ we denote by $[h]$ the equivalence class in the quotient group $\mathcal{H}^*/\mathbb{C}^*$ that contains $ h.$
\end{notation}

We emphasize that the quotient group $\mathcal{H}^*/\mathbb{C}^*$ is a \emph{torsion-free} abelian group.

A corollary of the Borel Lemma is the following.
\begin{prop}[see Proposition 4.5 in \cite{Fujimoto75}]      \label{prop:multi._independe._implies_algebrai._independe.}
Let $\eta_1,\dots,\eta_t $ be nowhere zero holomorphic functions on $\mathbb{C}^m $ and assume they are multiplicatively independent, namely,
 \[
    \eta_1^{n_1}\cdot\eta_2^{n_2}\cdots\eta_t^{n_t}\not\in \mathbb{C}^*
 \]
for any $(n_1,\dots,n_t)\in\mathbb{Z}^t\setminus\{(0,\dots,0)\}.$ If a polynomial $ P(X_1,\dots, X_t)\in\mathbb{C}[X_1,\dots,X_t]$ satisfies
 \[
    P(\eta_1,\dots,\eta_t)\equiv 0,
 \]
then $ P(X_1,\dots, X_t)$ is the zero polynomial.
\end{prop}

\paragraph{2.3\, Combinatorial lemmas.}

Let $ G $ be a torsion-free abelian group. Let $ A=(\alpha_1,\dots,\alpha_q)$ be a $ q$-tuple of elements in $ G.$ We denote by $\langle\alpha_1,\dots, \alpha_q\rangle$ the subgroup of $ G $ generated by $\alpha_1,\dots, \alpha_q.$ Because $ G $ is torsion-free, $\langle\alpha_1,\dots, \alpha_q\rangle$ is free and is of finite rank. We denote by {\rm rank}$\{\alpha_1,\dots, \alpha_q\}$ the rank of the free abelian group $\langle\alpha_1,\dots, \alpha_q\rangle.$

The following definition is due to Fujimoto \cite{Fujimoto75}.
\begin{defn}      \label{defn:property_(Pr,s)}
Let $(G,\cdot)$ be a torsion-free abelian group. Let $ q\geq r>s\geq 1 $ be integers. A $ q$-tuple $ A=(\alpha_1,\dots,\alpha_q)$ of elements in $ G $ is said to \emph{have the property} $(P_{r,s})$ if arbitrarily chosen $ r $ elements $\alpha_{l(1)},\dots,\alpha_{l(r)}$ in $ A $ ($ 1\leq l(1)<\dots<l(r)\leq q $) satisfy the condition that, for any $ s $ distinct indices $ i_1,\dots, i_s \in\{1,\dots,r\},$ there exist distinct indices $ j_1,\dots, j_s \in\{1,\dots,r\}$ with $\{j_1,\dots,j_s\}\neq \{i_1,\dots,i_s\}$ such that
 \[
    \alpha_{l(i_1)}\cdot\alpha_{l(i_2)}\cdots \alpha_{l(i_s)}= \alpha_{l(j_1)}\cdot\alpha_{l(j_2)}\cdots \alpha_{l(j_s)}.
 \]
\end{defn}

The following lemma is a refinement of Fujimoto's Lemma 2.6 in \cite{Fujimoto75}.
\begin{lem}      \label{lem:1stCombiLem_refinement}
Let $(G,\cdot)$ be a torsion-free abelian group. Let $ A=(\alpha_1,\dots,\alpha_q)$ be a $ q$-tuple of elements in $ G $ that has the property $(P_{r,s}),$ where $ q\geq r>s\geq 1.$ Then there exist $(q-r+2)$ distinct indices $ i_1,\dots,i_{q-r+2}\in\{1,\dots,q\}$ such that
 \[
    \alpha_{i_1}=\alpha_{i_2}=\dots=\alpha_{i_{q-r+2}}.
 \]
Furthermore, when $ r\leq 2s,$ the conclusion can be refined as follows:
 \begin{enumerate}[\rm (i)]
   \item if $ q\geq 2s-1,$ then there exist $ q-2(r-s)+2=:l $ distinct indices $ i_1,\dots,i_l\in\{1,\dots,q\}$ such that
      \[
         \alpha_{i_1}=\alpha_{i_2}=\dots=\alpha_{i_l};
      \]
   \item if $ q<2s-1,$ then there exist $ 2(q-r+2)$ distinct indices $ i_1,\dots,i_{q-r+2},$ $ j_1,\dots,j_{q-r+2}$ in $\{1,\dots,q\}$ such that
      \[
         \alpha_{i_1}=\alpha_{i_2}=\dots=\alpha_{i_{q-r+2}} \quad\mbox{and}\quad \alpha_{j_1}=\alpha_{j_2}=\dots=\alpha_{j_{q-r+2}}.
      \]
 \end{enumerate}
\end{lem}

\begin{remark}       \label{remk:P(r,s)_is_as_same_as_P(r,r-s)}
A $ q$-tuple $ A=(\alpha_1,\dots,\alpha_q)$ has the property $(P_{r,s})$ if and only if it has the property $(P_{r,r-s}).$ So the condition $ r\leq 2s $ is not a real restriction.
\end{remark}

The proof of Lemma \ref{lem:1stCombiLem_refinement} is given in Appendix A. Here we give a remark. Following Fujimoto's proof, we arrange the $\alpha_i $'s in a suitable order:
\[
   \alpha_{i_1}\prec \alpha_{i_2}\prec \dots \prec \alpha_{i_q}.
\]
We prove that
\[
   \alpha_{i_s}=\alpha_{i_{s+1}}=\dots=\alpha_{i_{q-(r-s)+1}}.
\]
Also we have
\[
   \alpha_{i_{r-s}}=\alpha_{i_{r-s+1}}=\dots=\alpha_{i_{q-s+1}}.
\]
Then, when $ r-s\leq s,$  by comparing the two numbers $(q-s+1)$ and $ s,$ the conclusion is naturally divided into two cases. Thus we get the conclusions of Lemma \ref{lem:1stCombiLem_refinement}.

The next lemma is a refinement of Fujimoto's Lemma 2.7 in \cite{Fujimoto75}.
\begin{lem}        \label{lem:2ndCombiLem_refinement}
Let $(G,\cdot)$ be a torsion-free abelian group. Let $ q\geq r>s\geq 1 $ be integers such that $ 2\leq r-s\leq s $ and $ q\geq 2s-1.$ Let $ A=(\alpha_1,\dots,\alpha_q)$ be a $ q$-tuple of elements in $ G $ that has the property $(P_{r,s}).$ Set $ l:=q-2(r-s)+2.$ Then, when $ r-s\geq 3,$ at least one of the following three cases occurs:
 \begin{enumerate}
   \item[$(\alpha)$] there exist $(l+1)$ distinct indices $ i_1,\dots,i_{l+1}\in\{1,\dots,q\}$ such that
      \[
         \alpha_{i_1}=\alpha_{i_2}=\dots=\alpha_{i_{l+1}};
      \]
   \item[$(\beta)$] there exist $(l+2)$ distinct indices $ i_1,\dots,i_{l+2}\in\{1,\dots,q\}$ such that
      \[
         \alpha_{i_1}=\alpha_{i_2}=\dots=\alpha_{i_l} \quad\mbox{and}\quad \alpha_{i_{l+1}}=\alpha_{i_{l+2}};
      \]
   \item[$(\gamma)$] there exist an integer $ k $ with $ 2\leq k\leq\min\{r-s,2s-r+2\}$ and $(l+k)$ distinct indices $ i_1,\dots,i_{l+k}\in\{1,\dots,q\}$ such that
      \[
         \alpha_{i_1}=\alpha_{i_2}=\dots=\alpha_{i_l} \quad\mbox{and}\quad \alpha_{i_{l+1}}\cdot\alpha_{i_{l+2}}\cdots\alpha_{i_{l+k}}=\alpha_{i_1}^k.
      \]
 \end{enumerate}
When $ r-s=2,$ one of the cases $(\alpha)$ and $(\gamma)$ occurs.
\end{lem}

The basic idea of the proof of Lemmas \ref{lem:1stCombiLem_refinement} and \ref{lem:2ndCombiLem_refinement} is essentially same as the idea of Fujimoto's proof in \cite[Section 2]{Fujimoto75}, though some new discussions are also needed. We shall give a proof of Lemmas \ref{lem:1stCombiLem_refinement} and \ref{lem:2ndCombiLem_refinement} in Appendix A.

To state another lemma, we need the following notation.
\begin{notation}
For elements $\alpha_1,\alpha_2,\dots,\alpha_q $, $\tilde{\alpha}_1,\tilde{\alpha}_2,\dots,\tilde{\alpha}_q $ in an abelian group $(G,\cdot),$ by the notation
 \[
    \alpha_1:\alpha_2:\dots:\alpha_q=\tilde{\alpha}_1:\tilde{\alpha}_2:\dots:\tilde{\alpha}_q,
 \]
we mean that $\alpha_i=\beta\tilde{\alpha}_i $ ($ 1\leq i\leq q $) for some element $\beta\in G.$
\end{notation}

Now we recall the following lemma.
\begin{lem}[see Lemma 3.6 in \cite{Fujimoto76} and \cite{zk23_ANote}]      \label{lem:3rdCombiLem}
Let $ s $ and $ q $ be integers with $ 2\leq s<q\leq 2s.$ Let $(G,\cdot)$ be a torsion-free abelian group. Let $ A=(\alpha_1,\dots,\alpha_q)$ be a $ q$-tuple of elements in $ G $ that has the property $(P_{q,s}),$ and assume that at least one $\alpha_i $ equals the unit element $ 1 $ of $ G.$ Then
 \begin{enumerate}[\rm (i)]
   \item {\rm rank}$\{\alpha_1,\dots,\alpha_q\}=:t\leq s-1;$
   \item if $ t=s-1,$ then $ q=2s $ and there is a basis $\{\beta_1,\dots,\beta_{s-1}\}$ of $\langle\alpha_1,\dots,\alpha_q\rangle$ such that the $\alpha_i $ are represented, after a suitable change of indices, as one of the following two types:
     \begin{itemize}
       \item[\rm (A)] $ s $ is odd and
            \[
               \alpha_1:\alpha_2:\dots:\alpha_{2s}= 1:1:\beta_1:\beta_1:\beta_2:\beta_2:\dots:\beta_{s-1}:\beta_{s-1};
            \]
       \item[\rm (B)] $\alpha_1:\alpha_2:\dots:\alpha_{2s}= 1:1:\dots:1:\beta_1:\dots:\beta_{s-1}: (\beta_1\cdots\beta_{a_1})^{-1}:(\beta_{a_1+1}\cdots\beta_{a_2})^{-1}:\dots:(\beta_{a_{k-1}+1}\cdots\beta_{a_k})^{-1},$
           \par where $ 0\leq k\leq s-1 $, $ 1\leq a_1<a_2<\dots<a_k\leq s-1,$ and the unit element $ 1 $ appears $(s+1-k)$ times in the right hand side.
     \end{itemize}
 \end{enumerate}
\end{lem}

We have the following easy observation.
\begin{observation}      \label{obsv:two_proportional_tuples_with_unit_generate_same_subgroup}
Let $(\alpha_1,\dots,\alpha_q)$ and $(\tilde{\alpha}_1,\dots,\tilde{\alpha}_q)$ be two $ q$-tuples of elements in a torsion-free abelian group $(G,\cdot).$ Assume that there exist indices $ i_0 $ and $ j_0 $ such that $\alpha_{i_0}=\tilde{\alpha}_{j_0}=1,$ and there is an element $\beta\in G $ such that $\alpha_i= \beta\tilde{\alpha}_i $ for any $ 1\leq i\leq q.$ Then
 \[
    \langle\alpha_1,\dots,\alpha_q \rangle= \langle\tilde{\alpha}_1,\dots,\tilde{\alpha}_q \rangle.
 \]
\end{observation}

Now let $ A=(\alpha_1,\dots,\alpha_q)$ be a $ q$-tuple that has the property $(P_{r,s})$ where $ q\geq r>s\geq 1 $ and $ r\leq 2s,$ and assume that at least one $\alpha_i $ equals the unit element $ 1.$ Set $ t:={\rm rank}\{\alpha_1,\dots,\alpha_q\}.$ By Lemma \ref{lem:1stCombiLem_refinement}, we may assume that $\alpha_{r-1}=\alpha_r=\dots=\alpha_q.$ Put $\alpha'_i:=\alpha_{r-1}^{-1}\cdot\alpha_i $ for each $ 1\leq i\leq q.$ Consider the new $ r$-tuple $(\alpha'_1,\dots,\alpha'_r)$ which also has the property $(P_{r,s}).$ By Observation \ref{obsv:two_proportional_tuples_with_unit_generate_same_subgroup}, we know that ${\rm rank}\{\alpha'_1,\dots,\alpha'_r\}=t.$ If $ s\geq 2,$ then Lemma \ref{lem:3rdCombiLem} tells us that $ t\leq s-1,$ and $ t=s-1 $ implies $ r=2s.$ If $ s=1,$ then $ r=2 $ and $ t=0.$ Whatever the case, we have $ t\leq s-1,$ and $ t=s-1 $ implies $ r=2s.$

The above discussion shows the following result.

\begin{lem}      \label{lem:3rdCombiLem_general_q-tuple}
Let $(G,\cdot)$ be a torsion-free abelian group. Let $ q\geq r>s\geq 1 $ be integers such that $ r\leq 2s.$ Let $ A=(\alpha_1,\dots,\alpha_q)$ be a $ q$-tuple of elements in $ G $ that has the property $(P_{r,s}),$ and assume at least one $\alpha_i $ equals the unit element $ 1 $ of $ G.$ Set $ t:={\rm rank}\{\alpha_1,\dots,\alpha_q\}.$ Then $ t\leq s-1,$ and $ t=s-1 $ implies $ r=2s.$
\end{lem}

\section{Extension of Theorem \ref{thm:(3n+2)thm_of_Fujimoto}}        \label{sec:Extension_of_(3n+2)thm}

We give first the following extension of Fujimoto's Theorem \ref{thm:(3n+2)thm_of_Fujimoto}.

\begin{thm}      \label{thm:q>=min(N+2n+2,max())}
Let $ N\geq n\geq 1 $ be integers. Let $\{H_j\}_{j=1}^q $ be hyperplanes in $\mathbb{P}^n $ in general position and $\{H'_j\}_{j=1}^q $ be hyperplanes in $\mathbb{P}^N $ in general position. Let $ f:\mathbb{C}^m\to\mathbb{P}^n $ and $ g:\mathbb{C}^m\to\mathbb{P}^N $ be meromorphic maps such that $ f(\mathbb{C}^m)\not\subseteq H_j $, $ g(\mathbb{C}^m)\not\subseteq H'_j $ and $ f^*(H_j)=g^*(H'_j)$ for any $ 1\leq j\leq q.$ If $ g $ is linearly non-degenerate and
 \[
    q\geq \min\big\{N+2n+2, \max\big\{2N+1,3n+2\big\}\big\},
 \]
then $ N=n $ and there are a projective linear transformation $ L $ of $\mathbb{P}^n $ and $(n+2)$ distinct indices $ k_1,\dots, k_{n+2}\in\{1,\dots,q\}$ such that $ f=L(g)$ and
 \[
    L(H'_{k_i})=H_{k_i}, \quad 1\leq i\leq n+2.
 \]
\end{thm}

\begin{proof}
Take reduced representations $(f_0,\dots,f_n)$ and $(g_0,\dots,g_N)$ of $ f $ and $ g,$ respectively. For each $ 1\leq j\leq q,$ take linear forms $ a^j_0X_0+\dots+a^j_nX_n $ and $ b^j_0Y_0+\dots+b^j_NY_N $ that define $ H_j $ and $ H'_j,$ respectively.

Define, for each $ 1\leq i\leq q,$
\begin{equation}       \label{equ:definition_of_hi}
   h_i:=\frac{(f,H_i)}{(g,H'_i)},
\end{equation}
where
\[
   (f,H_i)=a^i_0f_0+\dots+a^i_nf_n \quad\mbox{and}\quad (g,H'_i)=b^i_0g_0+\dots+b^i_Ng_N.
\]
By the assumption, we know that each $ h_i $ is a nowhere zero holomorphic function on $\mathbb{C}^m.$

We note that the condition on $ q $ implies $ q\geq N+n+2.$ Choose $(N+n+2)$ functions among the $ h_i$'s arbitrarily, say $ h_1,\dots,h_{N+n+2}.$

From definition, we get
\[
   a^i_0f_0+\dots+a^i_nf_n-b^i_0h_ig_0-\dots-b^i_Nh_ig_N=0, \quad 1\leq i\leq N+n+2.
\]
Eliminating $ f_0,\dots,f_n, g_0\dots,g_N $ from these equations, we conclude that
\[
   \det\big(a^i_0,\dots,a^i_n, b^i_0h_i,\dots,b^i_Nh_i;\, 1\leq i\leq N+n+2\big)\equiv 0.
\]
By the Laplace expansion formula, we deduce from the above equation the following identity:
\begin{equation}           \label{equ:sum(AIhI)=0}
   \sum_{1\leq i_1<i_2<\dots<i_{N+1}\leq N+n+2} A_{i_1,\dots,i_{N+1}}\cdot h_{i_1}\cdot h_{i_2}\cdots h_{i_{N+1}}\equiv 0,
\end{equation}
where
\begin{align*}
   A_{i_1,\dots,i_{N+1}}=\pm &\det\big(a^i_0,\dots,a^i_n;\, i\in\{1,\dots,N+n+2\}\setminus\{i_1,\dots,i_{N+1}\}\big)
\\ &\cdot\det\big(b^i_0,\dots,b^i_N;\, i\in\{i_1,\dots,i_{N+1}\}\big).
\end{align*}
Because $\{H_j\}_{j=1}^q $ and $\{H'_j\}_{j=1}^q $ are both in general position, the constants $ A_{i_1,\dots,i_{N+1}}$ are all nonzero. Thus by making use of the Borel Lemma (Theorem \ref{thm:Borel_Lemma}), we deduce from the above identity \eqref{equ:sum(AIhI)=0} that, for any $(N+1)$ distinct indices $ i_1,\dots,i_{N+1}\in\{1,\dots,N+n+2\},$ there exist distinct $ j_1,\dots,j_{N+1}\in\{1,\dots,N+n+2\}$ with $\{j_1,\dots,j_{N+1}\}\neq \{i_1,\dots,i_{N+1}\}$ such that
\[
   (h_{j_1}\cdot h_{j_2}\cdots h_{j_{N+1}})/(h_{i_1}\cdot h_{i_2}\cdots h_{i_{N+1}})\in\mathbb{C}^*,
\]
which gives (see Section \ref{sec:Preliminaries} for notations)
\[
   [h_{j_1}]\cdot [h_{j_2}]\cdots [h_{j_{N+1}}]=[h_{i_1}]\cdot [h_{i_2}]\cdots [h_{i_{N+1}}] (\in\mathcal{H}^*/\mathbb{C}^*).
\]

Using the Definition \ref{defn:property_(Pr,s)}, the above discussion in fact shows that the $ q$-tuple $([h_1],\dots,[h_q])$ whose elements are in the quotient group $\mathcal{H}^*/\mathbb{C}^*$ has the property $(P_{N+n+2,N+1}).$

Then by Lemma \ref{lem:1stCombiLem_refinement} and using the condition on $ q,$ we get that there are $(n+2)$ distinct indices $ k_1,\dots,k_{n+2}\in\{1,\dots,q\}$ such that
\[
   [h_{k_1}]=[h_{k_2}]=\dots=[h_{k_{n+2}}].
\]
In fact, if $ q\geq N+2n+2 $ then
\[
   q-(N+n+2)+2\geq n+2
\]
and the above conclusion follows from Lemma \ref{lem:1stCombiLem_refinement}. If $ q\geq\max\{2N+1,3n+2\},$ then
\[
   q-2(n+1)+2\geq n+2
\]
and the above conclusion follows from the refined conclusion (i) of Lemma \ref{lem:1stCombiLem_refinement}.

By choosing a new reduced representation of $ f $ if necessary, we may assume that $ h_{k_1}=:c_1 $, $ h_{k_2}=:c_2 $, $\dots$, $ h_{k_{n+2}}=:c_{n+2}$ are all constants.

In the following, if there is no confusion, we shall denote by $ f $, $ g $, $ H_i $ and $ H'_i $ ($ 1\leq i\leq q $) the row vectors $(f_0,\dots,f_n)$, $(g_0,\dots,g_N)$, $(a^i_0,\dots,a^i_n)$ and $(b^i_0,\dots,b^i_N)$ ($ 1\leq i\leq q $), respectively. Then, for any $ 1\leq i\leq q,$ $(f,H_i)$ equals the inner product $\langle f,H_i \rangle,$ and $(g,H'_i)$ equals $\langle g,H'_i \rangle.$

Since $ H_{k_1},H_{k_2},\dots,H_{k_{n+2}}$ are hyperplanes in $\mathbb{P}^n $ in general position, we can write
\[
   H_{k_{n+2}}=x_1H_{k_1}+x_2H_{k_2}+\dots+x_{n+1}H_{k_{n+1}},
\]
where $ x_1,x_2,\dots,x_{n+1}$ are nonzero constants.

We note that $\langle f,H_{k_i}\rangle=c_i\langle g,H'_{k_i}\rangle$ for any $ 1\leq i\leq n+2.$ So
\[
   \langle f,H_{k_{n+2}}\rangle=\sum_{i=1}^{n+1}x_i\langle f,H_{k_i}\rangle= \sum_{i=1}^{n+1}x_ic_i\langle g,H'_{k_i}\rangle= \langle g,\sum_{i=1}^{n+1}x_ic_iH'_{k_i}\rangle.
\]
Thus
\[
   c_{n+2}\langle g,H'_{k_{n+2}}\rangle=\langle g,\sum_{i=1}^{n+1}x_ic_iH'_{k_i}\rangle.
\]
It then follows that
\[
   \langle g,c_{n+2}H'_{k_{n+2}}-\sum_{i=1}^{n+1}x_ic_iH'_{k_i}\rangle=0.
\]
By the linear non-degeneracy of $ g,$ we conclude that
\begin{equation}      \label{equ:H'(n+2)-sum(xiciH'i)=0}
   c_{n+2}H'_{k_{n+2}}-\sum_{i=1}^{n+1}x_ic_iH'_{k_i}=0.
\end{equation}
The equation \eqref{equ:H'(n+2)-sum(xiciH'i)=0} shows that the $(n+2)$ vectors $ H'_{k_1}$, $ H'_{k_2}$, $\dots$, $ H'_{k_{n+2}}$ in $\mathbb{C}^{N+1}$ are linearly dependent over $\mathbb{C}.$ If $ N>n,$ then $ n+2\leq N+1 $ and \eqref{equ:H'(n+2)-sum(xiciH'i)=0} gives a contradiction because $\{H'_j\}$ are hyperplanes in $\mathbb{P}^N $ in general position. So it necessarily holds that $ N=n.$

Now the matrices
\[
   \left(\begin{array}{c}
       H_{k_1}
    \\ \vdots
    \\ H_{k_{n+1}}
   \end{array}\right)=:A
 \quad\mbox{and}\quad
   \left(\begin{array}{c}
       H'_{k_1}
    \\ \vdots
    \\ H'_{k_{n+1}}
   \end{array}\right)=:B
\]
are both invertible square matrices of order $(n+1).$

Denote by $ C $ the diagonal matrix ${\rm diag}(c_1,\dots,c_{n+1}).$ Then we have
\[
 \left(\begin{array}{c}
        f_0
     \\ \vdots
     \\ f_n
  \end{array}\right)=
 A^{-1}CB
 \left(\begin{array}{c}
        g_0
     \\ \vdots
     \\ g_n
  \end{array}\right).
\]
Put $\tilde{L}:=A^{-1}CB.$ We write the homogeneous coordinate of a point in $\mathbb{P}^n $ as a column vector. Then the correspondence
\[
 \left(\begin{array}{c}
        w_0
     \\ \vdots
     \\ w_n
  \end{array}\right)\mapsto
 \tilde{L}\left(\begin{array}{c}
        w_0
     \\ \vdots
     \\ w_n
  \end{array}\right)
\]
defines a projective linear transformation of $\mathbb{P}^n $ which is denoted by $ L.$ We see that $ f=L(g).$

For any $ 1\leq i\leq n+1,$ by
\[
   H'_{k_i}B^{-1}=(0,\dots,0,\underset{\underset{i-{\rm th}}{\uparrow}}{1},0,\dots,0),
\]
we deduce easily that $ L(H'_{k_i})=H_{k_i}.$

Because the $ H'_j$'s are hyperplanes in $\mathbb{P}^n $ in general position, we can write
\[
   H'_{k_{n+2}}=y_1H'_{k_1}+y_2H'_{k_2}+\dots+y_{n+1}H'_{k_{n+1}},
\]
where $ y_1,y_2,\dots,y_{n+1}$ are nonzero constants. Substituting this into \eqref{equ:H'(n+2)-sum(xiciH'i)=0}, we get
\[
   \sum_{i=1}^{n+1} (c_{n+2}y_i-x_ic_i)H'_{k_i}=0,
\]
which implies that
\[
   c_i=\frac{y_i}{x_i}\cdot c_{n+2}, \quad 1\leq i\leq n+1.
\]
Then easy computation tells us that $ L(H'_{k_{n+2}})=H_{k_{n+2}}.$ This finishes the proof of Theorem \ref{thm:q>=min(N+2n+2,max())}.
\end{proof}

Next we shall reduce the number of hyperplanes in the condition of Theorem \ref{thm:q>=min(N+2n+2,max())} under a stronger non-degeneracy condition of $ g.$

\begin{thm}       \label{thm:q>=max(3n+2,N+1.5n+1.5)}
Let $ N\geq n\geq 1 $ be integers. Let $\{H_j\}_{j=1}^q $, $\{H'_j\}_{j=1}^q $, $ f $ and $ g $ be as in Theorem {\rm \ref{thm:q>=min(N+2n+2,max())}}. If $ g(\mathbb{C}^m)$ is not contained in any hypersurface of degree $\leq 2 $ in $\mathbb{P}^N $ and
 \[
    q\geq \max\big\{3n+2, N+\frac{3}{2}n+\frac{3}{2}\big\},
 \]
then the conclusion of Theorem {\rm \ref{thm:q>=min(N+2n+2,max())}} holds.
\end{thm}

\begin{proof}
Let $(f_0,\dots,f_n)$, $(g_0,\dots,g_N)$, $ a^j_0X_0+\dots+a^j_nX_n $ and $ b^j_0Y_0+\dots+b^j_NY_N $ ($ 1\leq j\leq q $) be as in the proof of Theorem \ref{thm:q>=min(N+2n+2,max())}. For each $ 1\leq i\leq q,$ let $ h_i $ be defined as \eqref{equ:definition_of_hi}. Then each $ h_i $ is a nowhere zero holomorphic function on $\mathbb{C}^m.$

Since $ N+3n/2+3/2\geq N+n+2,$ the assumption on $ q $ implies that $ q\geq N+n+2.$ As in the proof of Theorem \ref{thm:q>=min(N+2n+2,max())}, we know that the $ q$-tuple $([h_1],\dots,[h_q])$ has the property $(P_{N+n+2,N+1}).$

Consider first the case when $ q\geq 2N+1.$ Since $ q-2(n+1)+2\geq n+2,$ the refined conclusion (i) of Lemma \ref{lem:1stCombiLem_refinement} tells us that there are $(n+2)$ distinct indices $ k_1,\dots,k_{n+2}\in\{1,\dots,q\}$ such that
\[
   [h_{k_1}]=[h_{k_2}]=\dots=[h_{k_{n+2}}].
\]
Then the same discussion as in the proof of Theorem \ref{thm:q>=min(N+2n+2,max())} shows the conclusion of Theorem \ref{thm:q>=min(N+2n+2,max())} holds.

So, it suffices to prove next that the case of $ q<2N+1 $ cannot happen.
Assume $ q<2N+1.$ We shall show a contradiction.

By the refined conclusion (ii) of Lemma \ref{lem:1stCombiLem_refinement}, we get there are $ 2(q-N-n)$ distinct indices $ i_0,\dots,i_{q-N-n-1}, j_0,\dots,j_{q-N-n-1}\in\{1,\dots,q\}$ such that
\begin{equation}           \label{equ:[hi0]=...=[hip]_and_[hj0]=...=[hjp]}
   [h_{i_0}]=[h_{i_1}]=\dots=[h_{i_{q-N-n-1}}] \quad\mbox{and}\quad [h_{j_0}]=[h_{j_1}]=\dots=[h_{j_{q-N-n-1}}].
\end{equation}
By $ q\geq N+3n/2+3/2,$ we have
\[
   2(q-N-n-1)\geq n+1.
\]
Let $ l $ be the largest integer that does not exceed $(n+1)/2,$ and set $\tilde{l}:=n+1-l.$ One checks easily that
\[
   1\leq l\leq \tilde{l}\leq q-N-n-1 \quad\mbox{and}\quad l+\tilde{l}=n+1.
\]
From \eqref{equ:[hi0]=...=[hip]_and_[hj0]=...=[hjp]}, we know that $ h_{i_k}/h_{i_0}=:c_k \,(1\leq k\leq l)$ and $ h_{j_k}/h_{j_0}=:\tilde{c}_k \,(1\leq k\leq \tilde{l})$ are all nonzero constants.

As in the proof of Theorem \ref{thm:q>=min(N+2n+2,max())}, sometimes we shall use the symbols $ f,g $, $ H_j, H'_j \,(1\leq j\leq q)$ to denote the corresponding row vectors.

Since $\{H_j\}$ are hyperplanes in $\mathbb{P}^n $ in general position, we can write
\[
   H_{i_0}=x_1H_{i_1}+\dots+x_lH_{i_l}+ x_{l+1}H_{j_1}+\dots+x_{n+1}H_{j_{\tilde{l}}}
\]
and
\[
   H_{j_0}=\tilde{x}_1H_{i_1}+\dots+\tilde{x}_lH_{i_l}+ \tilde{x}_{l+1}H_{j_1}+\dots+\tilde{x}_{n+1}H_{j_{\tilde{l}}},
\]
where $ x_1,\dots,x_{n+1}, \tilde{x}_1,\dots,\tilde{x}_{n+1}$ are nonzero constants.

We then have
\[
   \langle f,H_{i_0}\rangle=\sum_{k=1}^l x_k\langle f,H_{i_k}\rangle +\sum_{k=1}^{\tilde{l}} x_{l+k}\langle f,H_{j_k}\rangle.
\]
So
\begin{align*}
   \langle g,H'_{i_0}\rangle h_{i_0}=\langle f,H_{i_0}\rangle &=\sum_{k=1}^l x_kc_kh_{i_0}\langle g,H'_{i_k}\rangle +\sum_{k=1}^{\tilde{l}} x_{l+k}\tilde{c}_kh_{j_0}\langle g,H'_{j_k}\rangle
\\   &=\langle g,\sum_{k=1}^l x_kc_k H'_{i_k}\rangle h_{i_0} +\langle g,\sum_{k=1}^{\tilde{l}} x_{l+k}\tilde{c}_k H'_{j_k}\rangle h_{j_0},
\end{align*}
which implies
\[
   \langle g,H'_{i_0}-\sum_{k=1}^l x_kc_k H'_{i_k}\rangle h_{i_0}=\langle g,\sum_{k=1}^{\tilde{l}} x_{l+k}\tilde{c}_k H'_{j_k}\rangle h_{j_0}.
\]
Similarly, we have
\[
   \langle g,H'_{j_0}-\sum_{k=1}^{\tilde{l}} \tilde{x}_{l+k}\tilde{c}_k H'_{j_k}\rangle h_{j_0}=\langle g,\sum_{k=1}^l \tilde{x}_kc_k H'_{i_k}\rangle h_{i_0}.
\]
From the above two equations, we conclude
\begin{align}         \label{equ:implied_equation_of_g_when_q<2N+1}
   \langle g,H'_{i_0}-\sum_{k=1}^l x_kc_k H'_{i_k}\rangle &\cdot \langle g,H'_{j_0}-\sum_{k=1}^{\tilde{l}} \tilde{x}_{l+k}\tilde{c}_k H'_{j_k}\rangle \notag
\\ = &\langle g,\sum_{k=1}^{\tilde{l}} x_{l+k}\tilde{c}_k H'_{j_k}\rangle \cdot \langle g,\sum_{k=1}^l \tilde{x}_kc_k H'_{i_k}\rangle.
\end{align}
We note that the vectors
\[
   H'_{i_0}-\sum_{k=1}^l x_kc_k H'_{i_k},\quad H'_{j_0}-\sum_{k=1}^{\tilde{l}} \tilde{x}_{l+k}\tilde{c}_k H'_{j_k},\quad \sum_{k=1}^{\tilde{l}} x_{l+k}\tilde{c}_k H'_{j_k} \quad\mbox{and}\quad \sum_{k=1}^l \tilde{x}_kc_k H'_{i_k}
\]
are all nonzero vectors in $\mathbb{C}^{N+1}.$

Because both the left and right hand sides of \eqref{equ:implied_equation_of_g_when_q<2N+1} can be written as a homogeneous polynomial of degree 2 in $ g_0,\dots, g_N,$ we conclude by the non-degeneracy condition of $ g $ that \eqref{equ:implied_equation_of_g_when_q<2N+1} holds as an identity of polynomials in $ g_0,\dots, g_N.$ This tells us that the linear polynomial $\langle g,H'_{i_0}-\sum_{k=1}^l x_kc_k H'_{i_k}\rangle$ is a constant multiple of $\langle g,\sum_{k=1}^{\tilde{l}} x_{l+k}\tilde{c}_k H'_{j_k}\rangle,$ or is a constant multiple of $\langle g,\sum_{k=1}^l \tilde{x}_kc_k H'_{i_k}\rangle.$ The latter implies that $ H'_{i_0},H'_{i_1},\dots, H'_{i_l}$ are linearly dependent over $\mathbb{C},$ which contradicts the ``in general position'' assumption of $\{H'_j\}.$ So it is necessary that $ H'_{i_0}-\sum_{k=1}^l x_kc_k H'_{i_k}$ is a constant multiple of $\sum_{k=1}^{\tilde{l}} x_{l+k}\tilde{c}_k H'_{j_k},$ namely, there is a nonzero constant $ c $ such that
\[
   H'_{i_0}-\sum_{k=1}^l x_kc_k H'_{i_k}-c\sum_{k=1}^{\tilde{l}} x_{l+k}\tilde{c}_k H'_{j_k}=0.
\]
If $ N>n,$ then $ n+2\leq N+1 $ and the $(n+2)$ vectors $ H'_{i_0},H'_{i_1},\dots, H'_{i_l}$, $ H'_{j_1},\dots, H'_{j_{\tilde{l}}}$ in $\mathbb{C}^{N+1}$ are linearly independent over $\mathbb{C},$ which contradicts the above equation. Hence $ N=n.$

Thus we get
\[
   q\geq N+n+2=2N+2,
\]
which contradicts our assumption that $ q<2N+1.$ This shows the case of $ q<2N+1 $ cannot happen and the proof is thus finished.
\end{proof}

\section{Extension of Theorem \ref{thm:(3n+1)thm_of_Fujimoto}}     \label{sec:Extension_of_(3n+1)thm}

We shall prove in this section the following theorem which extends Fujimoto's Theorem \ref{thm:(3n+1)thm_of_Fujimoto}.
\begin{thm}      \label{thm:q>=max(3n+1,min(2N+1,N+1.5n+1.5))}
Let $ N\geq n\geq 1 $ be integers. Let $\{H_j\}_{j=1}^q $, $\{H'_j\}_{j=1}^q $, $ f:\mathbb{C}^m\to\mathbb{P}^n $ and $ g:\mathbb{C}^m\to\mathbb{P}^N $ be as in Theorem {\rm \ref{thm:q>=min(N+2n+2,max())}}. If $ g(\mathbb{C}^m)$ is not contained in any hypersurface of degree $\leq\min\{n+1,N-n+2\}$ in $\mathbb{P}^N $ and
 \[
    q\geq \max\big\{3n+1, \min\big\{2N+1,N+\frac{3}{2}n+\frac{3}{2}\big\}\big\},
 \]
then $ N=n $ and there is a projective linear transformation $ L $ of $\mathbb{P}^n $ such that $ f=L(g)$ and $ L $ satisfies the following condition: either
 \[
    L(H'_{k_i})=H_{k_i}, \quad 1\leq i\leq n+2,
 \]
for $(n+2)$ distinct indices $ k_1,\dots, k_{n+2}\in\{1,\dots,q\},$ or
 \[
    L(H'_{k_i})=H_{k_i}, \,\, 1\leq i\leq n+1, \quad\mbox{and}\quad L(H'_{k_{n+2}})=H_{k_{n+3}},\,\, L(H'_{k_{n+3}})=H_{k_{n+2}},
 \]
for $(n+3)$ distinct indices $ k_1,\dots, k_{n+3}\in\{1,\dots,q\}.$
\end{thm}

\begin{remark}        \label{remk:comparison_of_two_expressions}
For positive integers $ N $ and $ n $ with $ N\geq n,$ we set
 \[
    \alpha(n,N):=\max\big\{3n+1, N+\frac{3}{2}n+\frac{3}{2}\big\}
 \]
and
 \[
    \beta(n,N):=\max\big\{3n+1, \min\big\{2N+1,N+\frac{3}{2}n+\frac{3}{2}\big\}\big\}.
 \]
To compare $\alpha(n,N)$ and $\beta(n,N),$ we consider the following three cases.
 \begin{itemize}
   \item $ N\geq \frac{3}{2}n+\frac{1}{2}.$ In this case, $\beta(n,N)=\alpha(n,N).$
   \item $ N\leq \frac{3}{2}n-\frac{1}{2}.$ In this case, $\beta(n,N)=\alpha(n,N)=3n+1.$
   \item $\frac{3}{2}n-\frac{1}{2}<N<\frac{3}{2}n+\frac{1}{2}.$ In this case, it is necessary that $ n $ is even and $ N=\frac{3}{2}n.$ Thus $\beta(n,N)=3n+1 $ and $\alpha(n,N)=3n+\frac{3}{2}.$
 \end{itemize}
\end{remark}

\begin{proof}[Proof of Theorem {\rm\ref{thm:q>=max(3n+1,min(2N+1,N+1.5n+1.5))}}]
Let $(f_0,\dots,f_n)$, $(g_0,\dots,g_N)$, $ a^j_0X_0+\dots+a^j_nX_n $ and $ b^j_0Y_0+\dots+b^j_NY_N $ ($ 1\leq j\leq q $) be as in the proof of Theorem \ref{thm:q>=min(N+2n+2,max())}. For each $ 1\leq i\leq q,$ define $ h_i $ as \eqref{equ:definition_of_hi}. Then each $ h_i $ is a nowhere zero holomorphic function on $\mathbb{C}^m.$

One checks easily that the assumption on $ q $ implies $ q\geq N+n+2.$ As in the proof of Theorem \ref{thm:q>=min(N+2n+2,max())}, we know that the $ q$-tuple $([h_1],\dots,[h_q])$ has the property $(P_{N+n+2,N+1}).$

By assumption, we have
\[
   q\geq \min\big\{2N+1,N+\frac{3}{2}n+\frac{3}{2}\big\}.
\]
If $ q\geq N+3n/2+3/2,$ then the same argument as in the proof of Theorem \ref{thm:q>=max(3n+2,N+1.5n+1.5)} tells us that $ q\geq 2N+1,$ because $ g(\mathbb{C}^m)$ is not contained in any hypersurface of degree $\leq 2 $ in $\mathbb{P}^N.$ So we always have
\[
   q\geq 2N+1.
\]

Then by making use of Lemma \ref{lem:2ndCombiLem_refinement} and using the condition that $ q\geq 3n+1,$ we conclude that at least one of the following three cases occurs: (we renumber the indices for simplicity)
\begin{enumerate}
  \item[$(\alpha)$] $[h_1]=[h_2]=\dots=[h_{n+2}];$
  \item[$(\beta)$] $[h_1]=[h_2]=\dots=[h_{n+1}]$ \,\, and \,\, $[h_{n+2}]=[h_{n+3}];$
  \item[$(\gamma)$] $[h_1]=[h_2]=\dots=[h_{n+1}]$ \,\, and \,\, $[h_{n+2}]\cdot[h_{n+3}]\cdots[h_{n+1+k}]=[h_1]^k,$ where $ k $ is a positive integer such that $ 2\leq k\leq\min\{n+1,N-n+2\}.$
\end{enumerate}

If case ($\alpha$) occurs, then the same argument as in the proof of Theorem \ref{thm:q>=min(N+2n+2,max())} tells us that $ N=n $ and there is a projective linear transformation $ L $ of $\mathbb{P}^n $ such that $ f=L(g)$ and $ L(H'_i)=H_i $ for any $ 1\leq i\leq n+2.$

Now assume case ($\beta$) occurs. By choosing a new reduced representation of $ f,$ we assume that $ h_1=:c_1 $, $ h_2=:c_2 $, $\dots$, $ h_{n+1}=:c_{n+1}$ are all constants. We also know that $ h_{n+2}=c_{n+2}h_{n+3}$ for a nonzero constant $ c_{n+2}.$

In the following, we sometimes use the symbols $ f,g $, $ H_j, H'_j \,(1\leq j\leq q)$ to denote the corresponding row vectors as in the proof of Theorem \ref{thm:q>=min(N+2n+2,max())}.

Since the $\{H_j\}$ are hyperplanes in $\mathbb{P}^n $ in general position, we can write
\[
   H_{n+2}=x_{n+2,1}H_1+ x_{n+2,2}H_2+\dots+ x_{n+2,n+1}H_{n+1}
\]
and
\[
   H_{n+3}=x_{n+3,1}H_1+ x_{n+3,2}H_2+\dots+ x_{n+3,n+1}H_{n+1},
\]
where $ x_{n+2,i}$ and $ x_{n+3,i}$ ($ 1\leq i\leq n+1 $) are nonzero constants.

Since $\langle f,H_i\rangle=c_i\langle g,H'_i\rangle$ ($ 1\leq i\leq n+1 $), we have
\[
   \langle f,H_{n+2}\rangle=\sum_{i=1}^{n+1}x_{n+2,i}\langle f,H_i\rangle= \sum_{i=1}^{n+1}x_{n+2,i}c_i\langle g,H'_i\rangle= \langle g,\sum_{i=1}^{n+1}x_{n+2,i}c_iH'_i\rangle.
\]
Similarly, we have
\[
   \langle f,H_{n+3}\rangle=\langle g,\sum_{i=1}^{n+1}x_{n+3,i}c_iH'_i\rangle.
\]
We note that both $\sum_{i=1}^{n+1}x_{n+2,i}c_iH'_i $ and $\sum_{i=1}^{n+1}x_{n+3,i}c_iH'_i $ are nonzero vectors in $\mathbb{C}^{N+1}.$

Substituting the above two equalities into the equation $ h_{n+2}=c_{n+2}h_{n+3},$ we deduce that
\[
   \langle g,\sum_{i=1}^{n+1}x_{n+2,i}c_iH'_i\rangle\cdot\langle g,H'_{n+3}\rangle= c_{n+2}\cdot\langle g,\sum_{i=1}^{n+1}x_{n+3,i}c_iH'_i\rangle \cdot\langle g,H'_{n+2}\rangle.
\]
Because $ g(\mathbb{C}^m)$ is not contained in any hypersurface of degree $\leq 2,$ the above equation holds as an identity of polynomials in $ g_0,\dots,g_N.$ Since the linear polynomial $\langle g,H'_{n+2}\rangle$ cannot be a constant multiple of $\langle g,H'_{n+3}\rangle,$ we conclude that $\langle g,H'_{n+2}\rangle$ is a constant multiple of $\langle g,\sum_{i=1}^{n+1}x_{n+2,i}c_iH'_i\rangle.$ This means that there is a nonzero constant $ c $ such that
\[
   H'_{n+2}-c\sum_{i=1}^{n+1}x_{n+2,i}c_iH'_i=0,
\]
which implies that the $(n+2)$ vectors $ H'_1,\dots, H'_{n+2}$ are linearly dependent over $\mathbb{C}.$ The same reasoning as in the proof of Theorem \ref{thm:q>=min(N+2n+2,max())} shows $ N=n.$

Now using the same method as in the proof of Theorem \ref{thm:q>=min(N+2n+2,max())}, we conclude that there is a projective linear transformation $ L $ of $\mathbb{P}^n $ such that $ f=L(g)$ and $ L(H'_i)=H_i $ for any $ 1\leq i\leq n+2.$

Finally, we assume case ($\gamma$) occurs. By choosing a new reduced representation of $ f,$ we assume that $ h_1=:c_1 $, $ h_2=:c_2 $, $\dots$, $ h_{n+1}=:c_{n+1}$ are all constants and
\[
   h_{n+2}\cdot h_{n+3}\cdots h_{n+1+k}=:c_{n+2}
\]
is also a constant.

As before, we can write
\[
   H_j=x_{j,1}H_1+ x_{j,2}H_2+\dots+ x_{j,n+1}H_{n+1}, \quad n+2\leq j\leq n+1+k,
\]
where $ x_{j,1},\dots, x_{j,n+1}$ are nonzero constants.

Since $\langle f,H_i\rangle=c_i\langle g,H'_i\rangle$ ($ 1\leq i\leq n+1 $), we have
\[
   \langle f,H_j\rangle=\langle g,\sum_{i=1}^{n+1}x_{j,i}c_iH'_i\rangle, \quad n+2\leq j\leq n+1+k.
\]
We note that $\sum_{i=1}^{n+1}x_{j,i}c_iH'_i $ ($ n+2\leq j\leq n+1+k $) are all nonzero vectors in $\mathbb{C}^{N+1}.$

Substituting these equalities into the equation $ h_{n+2}\cdot h_{n+3}\cdots h_{n+1+k}=c_{n+2},$ we conclude that
\[
   \prod_{j=n+2}^{n+1+k} \langle g,\sum_{i=1}^{n+1}x_{j,i}c_iH'_i\rangle=c_{n+2}\cdot \prod_{j=n+2}^{n+1+k} \langle g,H'_j\rangle.
\]
Because $ g(\mathbb{C}^m)$ is not contained in any hypersurface of degree $\leq\min\{n+1,N-n+2\},$ the above equation holds as an identity of polynomials in $ g_0,\dots,g_N.$ Then we see that there is a bijection $\sigma:\{n+2,\dots,n+1+k\}\to\{n+2,\dots,n+1+k\}$ such that the linear polynomial $\langle g,\sum_{i=1}^{n+1}x_{j,i}c_iH'_i\rangle$ is a constant multiple of $\langle g,H'_{\sigma(j)}\rangle$ for any $ n+2\leq j\leq n+1+k.$ This means that there are nonzero constants $\tilde{c}_j\, (n+2\leq j\leq n+1+k)$ such that
\begin{equation}         \label{equ:vectors_equations_deduced_inproofof_(3n+1)extension}
   \sum_{i=1}^{n+1}x_{j,i}c_iH'_i-\tilde{c}_j H'_{\sigma(j)}=0, \quad n+2\leq j\leq n+1+k.
\end{equation}
The case of $ j=n+2 $ implies that the $(n+2)$ vectors $ H'_1,\dots, H'_{n+1}, H'_{\sigma(n+2)}$ are linearly dependent over $\mathbb{C}.$ The same reasoning as in the proof of Theorem \ref{thm:q>=min(N+2n+2,max())} shows $ N=n.$

Then we get $ k=2 $ and the above $\sigma$ is a bijection from $\{n+2,n+3\}$ to $\{n+2,n+3\}.$

As in the proof of Theorem \ref{thm:q>=min(N+2n+2,max())}, we write the homogeneous coordinate of a point in $\mathbb{P}^n $ as a column vector and let $\tilde{L}=A^{-1}CB,$ where
\[
   A=\left(\begin{array}{c}
        H_1
     \\ \vdots
     \\ H_{n+1}
    \end{array}\right),\quad
   B=\left(\begin{array}{c}
        H'_{1}
     \\ \vdots
     \\ H'_{n+1}
    \end{array}\right)
    \quad\mbox{and}\quad C={\rm diag}(c_1,\dots,c_{n+1}).
\]
The matrix $\tilde{L}$ defines a projective linear transformation $ L $ of $\mathbb{P}^n $ and this $ L $ satisfies the condition that $ f=L(g)$ and $ L(H'_i)=H_i $ for $ 1\leq i\leq n+1.$

We write
\[
   H'_j=y_{j,1}H'_1+ y_{j,2}H'_2+\dots+ y_{j,n+1}H'_{n+1}, \quad n+2\leq j\leq n+3,
\]
where $ y_{j,1},\dots, y_{j,n+1}$ are nonzero constants.

We conclude from \eqref{equ:vectors_equations_deduced_inproofof_(3n+1)extension} that, for any $ j\in\{n+2,n+3\},$
\[
   c_i=\tilde{c}_j\cdot\frac{y_{\sigma(j),i}}{x_{j,i}}, \quad 1\leq i\leq n+1.
\]
Then direct computation shows $ L(H'_{\sigma(j)})=H_j $ for any $ j\in\{n+2,n+3\}.$

If $\sigma$ is the identity map, then we see that $ L(H'_i)=H_i $ for any $ 1\leq i\leq n+3.$ If $\sigma$ is not the identity map, then we see that
\[
   L(H'_i)=H_i,\,\, 1\leq i\leq n+1, \quad\mbox{and}\quad L(H'_{n+2})=H_{n+3},\,\, L(H'_{n+3})=H_{n+2}.
\]
This completes the proof of Theorem \ref{thm:q>=max(3n+1,min(2N+1,N+1.5n+1.5))}.
\end{proof}

By putting $ N=n=m=1 $ and $ q=4,$ we get from Theorem \ref{thm:q>=max(3n+1,min(2N+1,N+1.5n+1.5))} the following result concerning meromorphic functions.
\begin{cor}           \label{cor:meromorphic_functions_sharing_two_families_of_4_distinct_points}
Let $\{a_j\}_{j=1}^4 $ and $\{b_j\}_{j=1}^4 $ be two families of distinct points in the Riemann sphere. If two nonconstant meromorphic functions $ f $ and $ g $ on $\mathbb{C}$ satisfy the condition that $ f^*(a_j)=g^*(b_j) (1\leq j\leq 4),$ then there is a M\"obius transformation $ L $ such that $ f=L(g)$ and $ L $ satisfies the following condition: either
 \[
    L(b_{k_1})=a_{k_1}, L(b_{k_2})=a_{k_2}, L(b_{k_3})=a_{k_3}
 \]
for distinct $ k_1,k_2,k_3\in\{1,2,3,4\},$ or
 \[
    L(b_{k_1})=a_{k_1}, L(b_{k_2})=a_{k_2}, L(b_{k_3})=a_{k_4}, L(b_{k_4})=a_{k_3}
 \]
for a permutation $(k_1,k_2,k_3,k_4)$ of $\{1,2,3,4\}.$
\end{cor}

\section{The Zariski closure of the image of $f\times g$ in $\mathbb{P}^n\times\mathbb{P}^N$}        \label{sec:Zariski_closure_Vfg}

This section is an extension of \cite[Section 5]{Fujimoto76}.

Let $ N\geq n\geq 1 $ be integers. Let $\{H_j\}_{j=1}^q $ be hyperplanes in $\mathbb{P}^n $ in general position and $\{H'_j\}_{j=1}^q $ be hyperplanes in $\mathbb{P}^N $ in general position with
\[
   q\geq N+n+2.
\]
Let $ f:\mathbb{C}^m\to\mathbb{P}^n $ and $ g:\mathbb{C}^m\to\mathbb{P}^N $ be meromorphic maps such that $ f(\mathbb{C}^m)\not\subseteq H_j $, $ g(\mathbb{C}^m)\not\subseteq H'_j $ and $ f^*(H_j)=g^*(H'_j)$ for any $ 1\leq j\leq q.$

Let $ f\times g $ be the holomorphic map of $\mathbb{C}^m\setminus\big(I(f)\cup I(g)\big)$ into $\mathbb{P}^n\times\mathbb{P}^N $ that is given by $(f\times g)(z)=(f(z),g(z))$ for $ z\in\mathbb{C}^m\setminus\big(I(f)\cup I(g)\big),$ where $ I(f)$ and $ I(g)$ are the indeterminacy loci of $ f $ and $ g,$ respectively.

\begin{defn}        \label{defn:Vfg}
We define $ V_{f\times g}$ to be the Zariski closure of the image of $ f\times g $ in $\mathbb{P}^n\times\mathbb{P}^N,$ namely, the intersection of all algebraic sets in $\mathbb{P}^n\times\mathbb{P}^N $ that contain the image of $ f\times g.$
\end{defn}

One sees easily the following.

\begin{prop}       \label{prop:Vfg_is_irreducible}
The algebraic set $ V_{f\times g}$ is irreducible.
\end{prop}

\begin{proof}
Assume $ V_{f\times g}$ is reducible. This means that there are algebraic sets $ V_1 $ and $ V_2 $ in $\mathbb{P}^n\times\mathbb{P}^N $ with $ V_1\subsetneqq V_{f\times g}$ and $ V_2\subsetneqq V_{f\times g}$ such that $ V_{f\times g}=V_1\cup V_2.$ Then $(f\times g)^{-1}(V_1)=:A_1 $ and $(f\times g)^{-1}(V_2)=:A_2 $ are both analytic subsets of $\mathbb{C}^m\setminus\big(I(f)\cup I(g)\big)$ and $\mathbb{C}^m\setminus\big(I(f)\cup I(g)\big)= A_1\cup A_2.$ The definition of $ V_{f\times g}$ implies that $ A_1 $ and $ A_2 $ are both proper analytic subsets of $\mathbb{C}^m\setminus\big(I(f)\cup I(g)\big).$ This gives a contradiction because $\mathbb{C}^m\setminus\big(I(f)\cup I(g)\big)$ is irreducible. Thus $ V_{f\times g}$ is necessarily irreducible.
\end{proof}

Let $(f_0,\dots,f_n)$, $(g_0,\dots,g_N)$, and $ a^j_0X_0+\dots+a^j_nX_n $, $ b^j_0Y_0+\dots+b^j_NY_N $ ($ 1\leq j\leq q $) be as in the proof of Theorem \ref{thm:q>=min(N+2n+2,max())}. For each $ 1\leq i\leq q,$ let $ h_i $ be defined as \eqref{equ:definition_of_hi}. Then each $ h_i $ is a nowhere zero holomorphic function on $\mathbb{C}^m.$ By choosing a new reduced representation of $ f $ if necessary, we may assume at least one $ h_i $ is constant.

Consider the $ q$-tuple $([h_1],\dots,[h_q])$ of elements in $\mathcal{H}^*/\mathbb{C}^*.$ As in the proof of Theorem \ref{thm:q>=min(N+2n+2,max())}, the $ q$-tuple $([h_1],\dots,[h_q])$ has the property $(P_{N+n+2,N+1}).$ We define
\[
   t:={\rm rank}\big\{[h_1],\dots,[h_q]\big\}.
\]

\begin{remark}      \label{remk:t_is_independent_of_representations}
The number $ t $ is independent of the different choices of reduced representations of $ f $ and $ g $ and linear forms that define hyperplanes $ H_j $ and $ H'_j,$ as long as at least one $ h_i $ is constant.
\end{remark}

By Lemma \ref{lem:3rdCombiLem} (i) (or Lemma \ref{lem:3rdCombiLem_general_q-tuple}), we have the following.
\begin{prop}      \label{prop:0<=t<=N}
$ 0\leq t\leq N.$
\end{prop}

Now we give the following main theorem of this section. Let notations be as above.

\begin{thm}         \label{thm:not(P(2s-(N-n),s))_implies_dimVfg<=N-s+t}
Suppose the $ q$-tuple $([h_1],\dots,[h_q])$ does not have the property $(P_{2s-(N-n),s})$ for some positive integer $ s $ with $ N-n+1\leq s\leq N+1.$ Then
 \[
    \max\{t, \dim V_{f\times g}\}\leq N-s+t.
 \]
\end{thm}

\begin{proof}
Clearly we have $ N-n+1\leq s\leq N $ and $ t\geq 1.$

Take $\eta_1,\dots,\eta_t $ in $\mathcal{H}^*$ such that $\{[\eta_1],\dots,[\eta_t]\}$ is a basis of the free abelian group $\langle [h_1],\dots,[h_q]\rangle.$ Then the functions $ h_i $ are represented as follows:
\begin{equation}        \label{equ:hi_represented_by_eta1...etat}
   h_i=c_i\cdot\eta_1^{l_{i,1}}\cdots\eta_t^{l_{i,t}}, \quad 1\leq i\leq q,
\end{equation}
where $ l_{i,1},\dots, l_{i,t}$ are integers and $ c_i $ are nonzero constants.

By assumption, there are $ h_{i_1},\dots,h_{i_{2s-(N-n)}}$ such that the $(2s-(N-n))$-tuple $([h_{i_1}],\dots,[h_{i_{2s-(N-n)}}])$ does not have the property $(P_{2s-(N-n),s}),$ where $ i_1,\dots,i_{2s-(N-n)}$ are distinct indices in $\{1,\dots,q\}.$ Without loss of generality, we assume the $(2s-(N-n))$-tuple $([h_1],\dots,[h_{2s-(N-n)}])$ does not have the property $(P_{2s-(N-n),s}).$

Put $ s_1:=s-(N-n).$ We list the following equalities and inequalities for readers' convenience:
\[
   s_1+s=2s-(N-n),\quad s_1+(N+1-s)=n+1,
\]
\[
   1\leq s_1\leq n,\quad 1\leq N+1-s\leq n, \quad (2s-(N-n))+(N+1-s)=s+n+1.
\]

By making a transformation of homogenous coordinates on $\mathbb{P}^n,$ we may assume that
\[
   (a^i_0,\dots,a^i_n)=(0,\dots,0,\underset{i-{\rm th}}{\underset{\uparrow}{1}},0,\dots,0), \quad 1\leq i\leq s_1,
\]
and
\[
   (a^i_0,\dots,a^i_n)=(0,\dots,0,\underset{(i-s)-{\rm th}}{\underset{\uparrow}{1}},0,\dots,0), \quad 2s-(N-n)+1\leq i\leq s+n+1.
\]
By making a transformation of homogenous coordinates on $\mathbb{P}^N,$ we may assume that
\[
   (b^i_0,\dots,b^i_N)=(0,\dots,0,\underset{i-{\rm th}}{\underset{\uparrow}{1}},0,\dots,0), \quad 1\leq i\leq s,
\]
and
\[
   (b^i_0,\dots,b^i_N)=(0,\dots,0,\underset{(i-s_1)-{\rm th}}{\underset{\uparrow}{1}},0,\dots,0), \quad 2s-(N-n)+1\leq i\leq s+n+1.
\]
After these two transformations, we know that
\[
   \det\big(a^i_0,\dots,a^i_{s_1-1};\, i=i_1,\dots,i_{s_1}\big)\neq 0
\]
for any $ s_1 $ distinct indices $ i_1,\dots,i_{s_1}\in\{1,\dots,2s-(N-n)\},$ and
\[
   \det\big(b^i_0,\dots,b^i_{s-1};\, i=j_1,\dots,j_s\big)\neq 0
\]
for any $ s $ distinct indices $ j_1,\dots,j_s\in\{1,\dots,2s-(N-n)\}.$ We also know that the functions $ f_0,\dots,f_n $, $ g_0,\dots,g_N $ are all not identically zero.

Now we claim that
\begin{equation}            \label{equ:det(a0,...,a(s1-1),b0hi,...,b(s-1)hi)neq0}
   \det\big(a^i_0,\dots,a^i_{s_1-1},b^i_0h_i,\dots,b^i_{s-1}h_i;\, i=1,\dots,2s-(N-n)\big)\not\equiv 0.
\end{equation}
In fact, if the left hand side of \eqref{equ:det(a0,...,a(s1-1),b0hi,...,b(s-1)hi)neq0} is identically zero, then using the Laplace expansion formula and the Borel Lemma as in the proof of Theorem \ref{thm:q>=min(N+2n+2,max())} we conclude $([h_1],\dots,[h_{2s-(N-n)}])$ has the property $(P_{2s-(N-n),s}),$ which contradicts the assumption. Thus we have \eqref{equ:det(a0,...,a(s1-1),b0hi,...,b(s-1)hi)neq0}.

From the definition of $ h_i,$ we get the following equations:
\[
   a^i_0f_0+\dots+a^i_nf_n-b^i_0h_ig_0-\dots-b^i_Nh_ig_N=0, \quad 1\leq i\leq 2s-(N-n),
\]
and
\begin{equation}        \label{equ:f(i-s-1)=hig(i-s1-1)}
   f_{i-s-1}=h_ig_{i-s_1-1}, \quad 2s-(N-n)+1\leq i\leq s+n+1.
\end{equation}
Substituting \eqref{equ:f(i-s-1)=hig(i-s1-1)} into the first $(2s-(N-n))$ equations, we get
\begin{equation}          \label{equ:a0f0+...+a(s1-1)f(s1-1)-b0hig0-...-b(s-1)hig(s-1)=Bi}
    a^i_0f_0+\dots+a^i_{s_1-1}f_{s_1-1}-b^i_0h_ig_0-\dots-b^i_{s-1}h_ig_{s-1}= B_i, \quad 1\leq i\leq 2s-(N-n),
\end{equation}
where
\[
   B_i=(b^i_sh_i-a^i_{s_1}h_{2s-(N-n)+1})g_s+\dots+(b^i_Nh_i-a^i_nh_{s+n+1})g_N.
\]

By \eqref{equ:det(a0,...,a(s1-1),b0hi,...,b(s-1)hi)neq0}, we can apply the Cramer's rule to obtain the following expressions of $ f_j\, (0\leq j\leq s_1-1)$ and $ g_j\, (0\leq j\leq s-1)$:
\[
   f_j=\frac{\tilde{P}_{j,s}(h_1,\dots,h_{s+n+1})g_s+\dots+\tilde{P}_{j,N}(h_1,\dots,h_{s+n+1})g_N}{\tilde{P}(h_1,\dots,h_{2s-(N-n)})}, \quad 0\leq j\leq s_1-1,
\]
\[
   g_j=\frac{\tilde{Q}_{j,s}(h_1,\dots,h_{s+n+1})g_s+\dots+\tilde{Q}_{j,N}(h_1,\dots,h_{s+n+1})g_N}{\tilde{P}(h_1,\dots,h_{2s-(N-n)})}, \quad 0\leq j\leq s-1,
\]
where $\tilde{P}$, $\tilde{P}_{j,s},\dots,\tilde{P}_{j,N}$, $\tilde{Q}_{j,s},\dots,\tilde{Q}_{j,N}$ are polynomials with constant coefficients and $\tilde{P}$ is not zero.

Combining the above expressions with \eqref{equ:f(i-s-1)=hig(i-s1-1)}, we get expressions of $ f_j/f_n.$ Then by substituting \eqref{equ:hi_represented_by_eta1...etat} into these expressions and also the above expressions of $ g_j,$ we get the following:
\begin{equation}         \label{equ:express_f0...f(s1-1)_by_gs...gN_etai}
   \frac{f_j}{f_n}=\frac{P_{j,s}(1,\eta_1,\dots,\eta_t)g_s+\dots+P_{j,N}(1,\eta_1,\dots,\eta_t)g_N}{P_0(1,\eta_1,\dots,\eta_t)g_N}, \quad 0\leq j\leq s_1-1,
\end{equation}
\begin{equation}         \label{equ:express_g0...g(s-1)_by_gs...gN_etai}
   g_j=\frac{Q_{j,s}(1,\eta_1,\dots,\eta_t)g_s+\dots+Q_{j,N}(1,\eta_1,\dots,\eta_t)g_N}{P(1,\eta_1,\dots,\eta_t)}, \quad 0\leq j\leq s-1,
\end{equation}
\begin{equation}         \label{equ:express_fs1...fn_by_gs...gN_etai}
   \frac{f_j}{f_n}=\frac{R_j^+(1,\eta_1,\dots,\eta_t)g_{j+N-n}}{R_j^-(1,\eta_1,\dots,\eta_t)g_N}, \quad s_1\leq j\leq n-1,
\end{equation}
where $ P_0,P\in\mathbb{C}[z_0,z_1,\dots,z_t]$ are homogeneous polynomials of degree $ d\geq 1,$ and each of the polynomials $ P_{j,s},\dots,P_{j,N}$, $ Q_{j,s},\dots,Q_{j,N}$ is either zero or a homogeneous polynomial of degree $ d,$ and $ R_j^+$ and $ R_j^-$ are monomials with $\deg R_j^+=\deg R_j^-.$

Now let $\eta$ be the holomorphic map of $\mathbb{C}^m $ into $\mathbb{P}^t $ that is given by $\eta(z)=[1:\eta_1(z):\dots:\eta_t(z)]$ for $ z\in\mathbb{C}^m,$ and let $ f\times g\times\eta$ be the holomorphic map of $\mathbb{C}^m\setminus\big(I(f)\cup I(g)\big)$ into $\mathbb{P}^n\times\mathbb{P}^N\times\mathbb{P}^t $ that is given by $(f\times g\times\eta)(z)=(f(z),g(z),\eta(z))$ for $ z\in\mathbb{C}^m\setminus\big(I(f)\cup I(g)\big).$ Let $ V_{f,g,\eta}$ be the Zariski closure of the image of $ f\times g\times\eta$ in $\mathbb{P}^n\times\mathbb{P}^N\times\mathbb{P}^t.$ The same argument as in the proof of Proposition \ref{prop:Vfg_is_irreducible} shows that the algebraic set $ V_{f,g,\eta}$ is irreducible.

Let $ V $ be the algebraic set in $\mathbb{P}^n\times\mathbb{P}^N\times\mathbb{P}^t $ which is defined as follows:
\begin{align*}
   V:=\big\{([x_0:\dots:x_n],[y_0:\dots:y_N],[z_0:\dots:z_t])\in\mathbb{P}^n\times\mathbb{P}^N\times\mathbb{P}^t\, |\hspace{4em}
 \\ x_jP_0(z_0,\dots,z_t)y_N-x_nP_{j,s}(z_0,\dots,z_t)y_s-\dots-x_nP_{j,N}(z_0,\dots,z_t)y_N=0,
 \\  0\leq j\leq s_1-1,
 \\ y_jP(z_0,\dots,z_t)-Q_{j,s}(z_0,\dots,z_t)y_s-\dots-Q_{j,N}(z_0,\dots,z_t)y_N=0,
 \\  0\leq j\leq s-1,
 \\ x_jR_j^-(z_0,\dots,z_t)y_N-x_nR_j^+(z_0,\dots,z_t)y_{j+N-n}=0, \quad s_1\leq j\leq n-1\big\}.
\end{align*}
Then by \eqref{equ:express_f0...f(s1-1)_by_gs...gN_etai}, \eqref{equ:express_g0...g(s-1)_by_gs...gN_etai} and \eqref{equ:express_fs1...fn_by_gs...gN_etai}, we see that
\[
   V \supseteq V_{f,g,\eta}.
\]
Let $ U $ be the subset of $ V $ that contains all points $([x_0:\dots:x_n],[y_0:\dots:y_N],[z_0:\dots:z_t])$ in $ V $ such that
\[
   x_n\neq 0, \quad y_N\neq 0, \quad z_0\neq 0, \quad P_0(z_0,\dots,z_t)\neq 0, \quad P(z_0,\dots,z_t)\neq 0,
\]
\[
   R_j^-(z_0,\dots,z_t)\neq 0, \quad s_1\leq j\leq n-1.
\]
Then $ U $ is a Zariski-open subset of $ V.$ Because $ f_n $, $ g_N $, $ P_0(1,\eta_1,\dots,\eta_t)$, $ P(1,\eta_1,\dots,\eta_t)$ and $ R_j^-(1,\eta_1,\dots,\eta_t)\, (s_1\leq j\leq n-1)$ are nonzero holomorphic functions on $\mathbb{C}^m,$ we see that
\[
   V_{f,g,\eta}\cap U\neq\emptyset.
\]

Let $\phi:U\to\mathbb{C}^{N-s+t}$ be the map
\[
   ([x_0:\dots:x_n],[y_0:\dots:y_N],[z_0:\dots:z_t])\mapsto \Big(\frac{y_s}{y_N},\dots,\frac{y_{N-1}}{y_N},\frac{z_1}{z_0},\dots,\frac{z_t}{z_0}\Big).
\]
From the definition of $ V $ and $ U,$ we easily check that $\phi$ is injective. Now since $ V_{f,g,\eta}\cap U $ is a nonempty Zariski-open subset of $ V_{f,g,\eta}$ and $ V_{f,g,\eta}$ is irreducible, we conclude from the injectivity of $\phi|_{V_{f,g,\eta}\cap U}:V_{f,g,\eta}\cap U\to\mathbb{C}^{N-s+t}$ that
\[
   \dim V_{f,g,\eta}\leq N-s+t.
\]

Let $\pi_{12}:\mathbb{P}^n\times\mathbb{P}^N\times\mathbb{P}^t\to\mathbb{P}^n\times\mathbb{P}^N $ be the projection onto the first two components and let $\pi_3:\mathbb{P}^n\times\mathbb{P}^N\times\mathbb{P}^t\to\mathbb{P}^t $ be the projection onto the third component. Since $\pi_{12}(V_{f,g,\eta})$ contains the image of $ f\times g,$ one sees that $\pi_{12}(V_{f,g,\eta})\supseteq V_{f\times g}$ and thus
\[
   \dim V_{f\times g}\leq \dim \pi_{12}(V_{f,g,\eta})\leq \dim V_{f,g,\eta}.
\]
Since $\eta_1,\dots,\eta_t $ are multiplicatively independent, we conclude from Proposition \ref{prop:multi._independe._implies_algebrai._independe.} that $\eta:\mathbb{C}^m\to\mathbb{P}^t $ is algebraically non-degenerate. It then follows from $\pi_3(V_{f,g,\eta})\supseteq \eta(\mathbb{C}^m)$ that $\pi_3(V_{f,g,\eta})=\mathbb{P}^t.$ Thus
\[
   t=\dim \pi_3(V_{f,g,\eta})\leq \dim V_{f,g,\eta}.
\]

Now it follows from the above inequalities that
\[
   \max\{t, \dim V_{f\times g}\}\leq N-s+t.
\]
This proves Theorem \ref{thm:not(P(2s-(N-n),s))_implies_dimVfg<=N-s+t}.
\end{proof}

With the help of Lemma \ref{lem:3rdCombiLem} (i), we shall give the following corollary of Theorem \ref{thm:not(P(2s-(N-n),s))_implies_dimVfg<=N-s+t}.
\begin{cor}             \label{cor:dimVfg<=N_always}
$ V_{f\times g}$ is always of dimension $\leq N.$
\end{cor}

\begin{proof}
We consider the following two cases separately.

\noindent$\bullet$ $ t\geq N-n+1.$

In this case, we have $ N-n+1\leq t\leq N.$ If the $ q$-tuple $([h_1],\dots,[h_q])$ has the property $(P_{2t-(N-n),t}),$ then Lemma \ref{lem:3rdCombiLem} (i) (or Lemma \ref{lem:3rdCombiLem_general_q-tuple}) implies that $ t\leq t-1,$ which is a contradiction. Thus the $ q$-tuple $([h_1],\dots,[h_q])$ does not have the property $(P_{2t-(N-n),t}).$ Then by putting $ s=t $ in Theorem \ref{thm:not(P(2s-(N-n),s))_implies_dimVfg<=N-s+t}, we get
\[
   \dim V_{f\times g}\leq N.
\]
Corollary \ref{cor:dimVfg<=N_always} is thus proved in this case.

\noindent$\bullet$ $ t\leq N-n.$

Assume $\dim V_{f\times g}>N.$

We claim the $ q$-tuple $([h_1],\dots,[h_q])$ has the property $(P_{N-n+2,N-n+1}).$ Indeed, if $([h_1],\dots,[h_q])$ does not have the property $(P_{N-n+2,N-n+1}),$ noticing that $ N-n+2=2(N-n+1)-(N-n),$ we can put $ s=N-n+1 $ in Theorem \ref{thm:not(P(2s-(N-n),s))_implies_dimVfg<=N-s+t} and conclude
\[
   \dim V_{f\times g}\leq N-(N-n+1)+t\leq (n-1)+(N-n)=N-1,
\]
which contradicts our assumption.

Now applying Lemma \ref{lem:1stCombiLem_refinement} and noticing that $ q-(N-n+2)+2\geq 2n+2,$ we know that there are $(2n+2)$ distinct indices $ i_1,\dots,i_{2n+2}\in\{1,\dots,q\}$ such that
\[
   [h_{i_1}]=[h_{i_2}]=\dots=[h_{i_{2n+2}}].
\]
By choosing a new reduced representation of $ f,$ we assume that $ h_{i_1}=:c_1 $, $ h_{i_2}=:c_2 $, $\dots$, $ h_{i_{n+1}}=:c_{n+1}$ are all constants. Then we see easily that there exist a projective linear transformation $ A $ of $\mathbb{P}^n $ and a projective linear transformation $ B $ of $\mathbb{P}^N $ such that the image of $ A(f)\times B(g)$ is contained in
\[
   \big\{([x_0:\dots:x_n],[y_0:\dots:y_N])\in\mathbb{P}^n\times\mathbb{P}^N \,|\, x_ky_0-x_0y_k=0, \, k=1,\dots,n\big\}
\]
which is an algebraic set in $\mathbb{P}^n\times\mathbb{P}^N $ of dimension $ N.$ This implies
\[
   \dim V_{f\times g}\leq N,
\]
which contradicts our assumption.

Hence we also get $\dim V_{f\times g}\leq N $ in this case.
\end{proof}

\section{The equality of dimensions of target spaces when $g$ is algebraically non-degenerate}       \label{sec:N=n_when_g_algebraically_nondegenerate}

We shall prove in this section the following theorem.
\begin{thm}      \label{thm:N=n_when_g_algebraically_nondegenerate}
Let $ N\geq n\geq 1 $ be integers. Let $\{H_j\}_{j=1}^q $ be hyperplanes in $\mathbb{P}^n $ in general position and $\{H'_j\}_{j=1}^q $ be hyperplanes in $\mathbb{P}^N $ in general position with $ q\geq N+n+2.$ Let $ f:\mathbb{C}^m\to\mathbb{P}^n $ and $ g:\mathbb{C}^m\to\mathbb{P}^N $ be meromorphic maps such that $ f(\mathbb{C}^m)\not\subseteq H_j $, $ g(\mathbb{C}^m)\not\subseteq H'_j $ and $ f^*(H_j)=g^*(H'_j)$ for any $ 1\leq j\leq q.$ If $ g $ is algebraically non-degenerate, then $ N=n.$
\end{thm}

As a matter of fact, we shall prove the following stronger result.
\begin{prop}       \label{prop:g_linea._nondegene._dimVfg>=N_imply_N=n}
Let $ N $, $ n $, $\{H_j\}_{j=1}^q $, $\{H'_j\}_{j=1}^q $, $ f $ and $ g $ be as in Theorem {\rm\ref{thm:N=n_when_g_algebraically_nondegenerate}}. Let $ V_{f\times g}$ be as in Definition {\rm\ref{defn:Vfg}}. If $ g $ is linearly non-degenerate and $\dim V_{f\times g}\geq N,$ then $ N=n.$
\end{prop}

The following proposition shows Proposition \ref{prop:g_linea._nondegene._dimVfg>=N_imply_N=n} implies Theorem \ref{thm:N=n_when_g_algebraically_nondegenerate}.
\begin{prop}       \label{prop:g_algebrai._nondegene._implies_dimVfg>=N}
Let $ f:\mathbb{C}^m\to\mathbb{P}^n $ and $ g:\mathbb{C}^m\to\mathbb{P}^N $ be meromorphic maps and let $ V_{f\times g}$ be as in Definition {\rm\ref{defn:Vfg}}. If $ g $ is algebraically non-degenerate, then $\dim V_{f\times g}\geq N.$
\end{prop}
\begin{proof}
Let $\pi_2:\mathbb{P}^n\times\mathbb{P}^N\to\mathbb{P}^N $ be the projection onto the second component. Since $\pi_2(V_{f\times g})$ contains the image of $ g,$ we conclude from the algebraic non-degeneracy of $ g $ that $\pi_2(V_{f\times g})=\mathbb{P}^N.$ Thus
\[
   \dim V_{f\times g}\geq \dim\pi_2(V_{f\times g})=N,
\]
which proves Proposition \ref{prop:g_algebrai._nondegene._implies_dimVfg>=N}.
\end{proof}

\begin{proof}[Proof of Proposition {\rm\ref{prop:g_linea._nondegene._dimVfg>=N_imply_N=n}}]
Let $(f_0,\dots,f_n)$, $(g_0,\dots,g_N)$ and $ a^j_0X_0+\dots+a^j_nX_n $, $ b^j_0Y_0+\dots+b^j_NY_N $ ($ 1\leq j\leq q $) be as in Section \ref{sec:Zariski_closure_Vfg}. Define $ h_i $ as \eqref{equ:definition_of_hi} for each $ 1\leq i\leq q.$ Then each $ h_i $ is a nowhere zero holomorphic function on $\mathbb{C}^m.$ By choosing a new reduced representation of $ f,$ we assume at least one $ h_i $ is constant. As in the proof of Theorem \ref{thm:q>=min(N+2n+2,max())}, the $ q$-tuple $([h_1],\dots,[h_q])$ has the property $(P_{N+n+2,N+1}).$ It then follows from Lemma \ref{lem:3rdCombiLem} (i) (or Lemma \ref{lem:3rdCombiLem_general_q-tuple}) that
\[
   {\rm rank}\big\{[h_1],\dots,[h_q]\big\}=:t\leq N.
\]

To show $ N=n,$ we consider the following two cases separately.

\noindent$\bullet$ $ t<N-n.$

We claim the $ q$-tuple $([h_1],\dots,[h_q])$ has the property $(P_{N-n+2,N-n+1}).$ Indeed, if $([h_1],\dots,[h_q])$ does not have the property $(P_{N-n+2,N-n+1}),$ noticing that $ N-n+2=2(N-n+1)-(N-n),$ we can put $ s=N-n+1 $ in Theorem \ref{thm:not(P(2s-(N-n),s))_implies_dimVfg<=N-s+t} and conclude
\[
   \dim V_{f\times g}\leq N-(N-n+1)+t<(n-1)+(N-n)=N-1,
\]
which contradicts the assumption that $\dim V_{f\times g}\geq N.$

Now applying Lemma \ref{lem:1stCombiLem_refinement} and noticing that $ q-(N-n+2)+2\geq 2n+2,$ we know that there are $(2n+2)$ distinct indices $ i_1,\dots,i_{2n+2}\in\{1,\dots,q\}$ such that
\[
   [h_{i_1}]=[h_{i_2}]=\dots=[h_{i_{2n+2}}].
\]
Since $ g $ is assumed to be linearly non-degenerate, we conclude by using the same method as in the proof of Theorem \ref{thm:q>=min(N+2n+2,max())} that $ N=n.$

\noindent$\bullet$ $ t\geq N-n.$

In this case, we have $ N-n+1\leq t+1\leq N+1.$ If the $ q$-tuple $([h_1],\dots,[h_q])$ does not have the property $(P_{2(t+1)-(N-n),t+1}),$ then by putting $ s=t+1 $ in Theorem \ref{thm:not(P(2s-(N-n),s))_implies_dimVfg<=N-s+t} we get
\[
   \dim V_{f\times g}\leq N-(t+1)+t=N-1,
\]
which contradicts the assumption that $\dim V_{f\times g}\geq N.$ Therefore the $ q$-tuple $([h_1],\dots,[h_q])$ necessarily has the property $(P_{2(t+1)-(N-n),t+1}).$

Now applying the conclusion (ii) of Lemma \ref{lem:3rdCombiLem} (or Lemma \ref{lem:3rdCombiLem_general_q-tuple}), we get
\[
   2(t+1)-(N-n)=2(t+1),
\]
which implies $ N=n.$

The proof of Proposition \ref{prop:g_linea._nondegene._dimVfg>=N_imply_N=n} is now completed.
\end{proof}

We guess Theorem \ref{thm:N=n_when_g_algebraically_nondegenerate} is optimal in the sense that the number ``$ N+n+2 $'' cannot be replaced by ``$ N+n+1 $''. When $ n=1,$ this is true as is showed in the following example.
\begin{example}        \label{example:N>n=1_q=N+2_g_algebrai._nondege.}
Let $ N\geq 2 $ be any integer. Let $ G_1,\dots, G_N $ be holomorphic functions on $\mathbb{C}$ which are given by $ G_j(z)=z^j $, $ 1\leq j\leq N,$ respectively. Let $ a_1,\dots,a_N $ be distinct nonzero complex numbers. Using the Borel Lemma, we know that $ 1+e^{G_1}+\dots+e^{G_N}$ and $ 1+a_1e^{G_1}+\dots+a_Ne^{G_N}$ are both nonzero holomorphic functions on $\mathbb{C}.$ Take a nonzero holomorphic function $ F $ on $\mathbb{C}$ such that
 \[
    f_0:=\frac{1+e^{G_1}+\dots+e^{G_N}}{F} \quad\mbox{and}\quad f_1:=\frac{1+a_1e^{G_1}+\dots+a_Ne^{G_N}}{F}
 \]
are holomorphic functions on $\mathbb{C}$ without common zeros. Using the Borel Lemma, we see the meromorphic function $ f:=f_1/f_0 $ is nonconstant. Let $ g:\mathbb{C}\to\mathbb{P}^N $ be the holomorphic map that has a reduced representation $(g_0,\dots,g_N),$ where
 \[
    g_0=-f_0 \quad\mbox{and}\quad g_j=e^{G_j}(f_1-a_jf_0), \quad 1\leq j\leq N.
 \]
Then $ g $ is algebraically non-degenerate (see Appendix B). We also note that $ g_0+\dots+g_N=-f_1.$ Let $ a_{N+1}=\infty $ and $ a_{N+2}=0 $ be points in the Riemann sphere $\mathbb{C}\cup\{\infty\}.$ Let $ H'_1=\{Y_1=0\}$, $ H'_2=\{Y_2=0\}$, $\dots$, $ H'_N=\{Y_N=0\}$, $ H'_{N+1}=\{Y_0=0\}$, $ H'_{N+2}=\{Y_0+\cdots+Y_N=0\}$ be hyperplanes in $\mathbb{P}^N.$ Clearly, $ H'_1,\dots, H'_{N+2}$ are in general position. The holomorphic map $ g:\mathbb{C}\to\mathbb{P}^N $ and the meromorphic function $ f $ on $\mathbb{C}$ satisfy the condition that
 \[
    f^*(a_j)=g^*(H'_j), \quad 1\leq j\leq N+2.
 \]
\end{example}

\section{Open questions and further remarks}         \label{sec:Open_questions_further_remarks}

In view of Theorem \ref{thm:N=n_when_g_algebraically_nondegenerate}, the task of extending Fujimoto's Theorem \ref{thm:(2n+3)thm_of_Fujimoto} becomes answering the following question.
\begin{question}      \label{question:(2n+3)Hj_H'j}
Let $\{H_j\}_{j=1}^{2n+3}$ and $\{H'_j\}_{j=1}^{2n+3}$ be two families of hyperplanes in $\mathbb{P}^n $ and each of them is in general position. Let $ f $ and $ g $ be meromorphic maps of $\mathbb{C}^m $ into $\mathbb{P}^n $ such that $ f(\mathbb{C}^m)\not\subseteq H_j $, $ g(\mathbb{C}^m)\not\subseteq H'_j $ and $ f^*(H_j)=g^*(H'_j)$ for any $ 1\leq j\leq 2n+3.$ Suppose that $ g $ is algebraically non-degenerate. Is there a projective linear transformation $ L $ of $\mathbb{P}^n $ such that $ f=L(g)$ and $ L $ satisfies certain condition which forces $ L $ to be the identity map when $\{H_j\}_{j=1}^{2n+3}= \{H'_j\}_{j=1}^{2n+3}$?
\end{question}

Consider the case where the number of hyperplanes is $(2n+2)$ in each family and assume $ g $ is algebraically non-degenerate. Let notations be as in Section \ref{sec:Zariski_closure_Vfg}. By making use of Theorem \ref{thm:not(P(2s-(N-n),s))_implies_dimVfg<=N-s+t}, we see that the $(2n+2)$-tuple $([h_1],\dots,[h_{2n+2}])$ has the property $(P_{2t+2,t+1}).$ Then using Lemma \ref{lem:3rdCombiLem}, we conclude that the $[h_i]$ are represented, after a suitable change of indices, as one of the following two cases:
\begin{itemize}
  \item[\rm (a)] $[h_1]:[h_2]:\dots:[h_{2n+2}]= 1:1:\dots:1:1:\beta_1:\beta_1:\dots:\beta_t:\beta_t;$
  \item[\rm (b)] $[h_1]:[h_2]:\dots:[h_{2n+2}]= 1:1:\dots:1:\beta_1:\dots:\beta_t: (\beta_1\cdots\beta_{a_1})^{-1}:(\beta_{a_1+1}\cdots\beta_{a_2})^{-1}:\dots:(\beta_{a_{k-1}+1}\cdots\beta_{a_k})^{-1},$
           \par where $ 0\leq k\leq t $, $ 1\leq a_1<a_2<\dots<a_k\leq t.$
\end{itemize}
When $\{H_j\}=\{H'_j\},$ Fujimoto shows that if case (a) occurs then $ f=g $ (see \cite{Fujimoto76}), and if case (b) occurs then $ a_k=t $ (see \cite{Fujimoto76,Fujimoto78}). Using the property $ a_k=t $ in case (b), Fujimoto proves Theorem \ref{thm:(2n+3)thm_of_Fujimoto}. However, when $\{H_j\}\neq\{H'_j\},$ it seems that essential difficulties exist in the study of both case (a) and case (b).

We note that Question \ref{question:(2n+3)Hj_H'j} has an affirmative answer when $ n=1 $ or $ n=2.$ This can be seen by putting $ N=n=1 $ and $ q=5 $ in Theorem \ref{mainthm:q>=N+2n+2}, and putting $ N=n=2 $ and $ q=7 $ in Theorem \ref{mainthm:q>=max(3n+1,N+1.5n+1.5)}.

We guess the number ``$ N+n+2 $'' in Theorem \ref{thm:N=n_when_g_algebraically_nondegenerate} cannot be replaced by ``$ N+n+1 $''. So we ask the following question.
\begin{question}        \label{question:N>n_q=N+n+1}
For any positive integers $ N $ and $ n $ with $ N>n,$ are there an algebraically non-degenerate holomorphic map $ g:\mathbb{C}\to\mathbb{P}^N,$ a linearly non-degenerate holomorphic map $ f:\mathbb{C}\to\mathbb{P}^n,$ hyperplanes $\{H'_j\}_{j=1}^{N+n+1}$ in $\mathbb{P}^N $ in general position, and hyperplanes $\{H_j\}_{j=1}^{N+n+1}$ in $\mathbb{P}^n $ in general position such that $ f^*(H_j)=g^*(H'_j)$ for any $ 1\leq j\leq N+n+1 $?
\end{question}

We note that Example \ref{example:N>n=1_q=N+2_g_algebrai._nondege.} gives an affirmative answer for Question \ref{question:N>n_q=N+n+1} in the case of $ n=1.$

\appendix

\stepcounter{section}
\section*{Appendix \Alph{section}\hspace{0.73em} Proof of Lemmas \ref{lem:1stCombiLem_refinement} and \ref{lem:2ndCombiLem_refinement}}
\addcontentsline{toc}{section}{Appendix \Alph{section}\,\, Proof of Lemmas \ref{lem:1stCombiLem_refinement} and \ref{lem:2ndCombiLem_refinement}}

Let $(G,\cdot)$ be a torsion-free abelian group and let $ A=(\alpha_1,\dots,\alpha_q)$ be a $ q$-tuple of elements in $ G $ that has the property $(P_{r,s}),$ where $ q\geq r>s\geq 1.$ Since $ G $ is torsion-free, the subgroup $\langle\alpha_1,\dots,\alpha_q\rangle$ is free. We take a basis $\{\eta_1,\dots,\eta_t\}$ of $\langle\alpha_1,\dots,\alpha_q\rangle.$ Then every $\alpha_i $ ($ 1\leq i\leq q $) is uniquely represented as follows:
\begin{equation}     \label{equ:representation_of_alpha_i}
   \alpha_i=\eta_1^{l(i,1)}\cdot\eta_2^{l(i,2)}\cdots \eta_t^{l(i,t)},
\end{equation}
where $ l(i,1),\dots, l(i,t)$ are integers.

Take integers $ p_1,\dots,p_t $ such that the integers $ l_i $ ($ 1\leq i\leq q $), which are defined as follows:
\[
   l_i:=l(i,1)\cdot p_1+ \dots+ l(i,t)\cdot p_t,
\]
satisfy the condition that, for any indices $ i_1,\dots, i_k, j_1,\dots, j_k\in\{1,\dots,q\}$ with $ 1\leq k\leq s,$
\begin{equation}       \label{equ:sum(li)=sum(lj)_iff_prod(alphai)=prod(alphaj)}
   l_{i_1}+\dots+l_{i_k}=l_{j_1}+\dots+l_{j_k} \quad\mbox{if and only if}\quad \alpha_{i_1}\cdot\cdots\cdot\alpha_{i_k}= \alpha_{j_1}\cdot\cdots\cdot\alpha_{j_k}.
\end{equation}
In fact, we may put $ p_1=1,$ and once $ p_1,\dots, p_{u-1}$ are chosen, we then put
\[
   p_u=2s\sum_{\tau=1}^{u-1}\Big(\max_{1\leq i\leq q}|l(i,\tau)|\cdot p_{\tau}\Big)+1.
\]
For such chosen $ p_1,\dots, p_t,$ one checks easily that the integers $ l_i $ defined as above satisfy the condition that, for any $ i_1,\dots, i_k, j_1,\dots, j_k \in\{1,\dots,q\}$ with $ 1\leq k\leq s,$
\[
   l_{i_1}+\dots+l_{i_k}=l_{j_1}+\dots+l_{j_k}
\]
if and only if
\[
   l(i_1,\tau)+\dots+l(i_k,\tau)=l(j_1,\tau)+\dots+l(j_k,\tau) \quad \forall\, 1\leq\tau\leq t.
\]
By \eqref{equ:representation_of_alpha_i}, the latter statement is equivalent to
\[
   \alpha_{i_1}\cdot\cdots\cdot\alpha_{i_k}= \alpha_{j_1}\cdot\cdots\cdot\alpha_{j_k}.
\]
So the integers $ l_i $ ($ 1\leq i\leq q $) satisfy the condition \eqref{equ:sum(li)=sum(lj)_iff_prod(alphai)=prod(alphaj)}.

In particular, we know that $ l_i=l_j $ if and only if $\alpha_i=\alpha_j $ for any $ i,j\in\{1,\dots,q\}.$ We also know that, for any $ i_1,\dots, i_s, j_1,\dots, j_s \in\{1,\dots,q\},$
\[
   l_{i_1}+\dots+l_{i_s}=l_{j_1}+\dots+l_{j_s} \quad\mbox{if and only if}\quad \alpha_{i_1}\cdot\cdots\cdot\alpha_{i_s}= \alpha_{j_1}\cdot\cdots\cdot\alpha_{j_s}.
\]

Changing the indices if necessary, we may assume that
\[
   l_1\leq l_2\leq \cdots\leq l_q.
\]

Choose the following $ r $ elements among the $\alpha_i $'s:
\[
   \alpha_1,\alpha_2,\dots,\alpha_s, \alpha_{q-(r-s)+1},\alpha_{q-(r-s)+2},\dots,\alpha_q.
\]
Assume that $ l_s<l_{q-(r-s)+1}.$ Consider the $ s $ elements $\alpha_1,\alpha_2,\dots,\alpha_s.$ By assumption, there are indices $ j_1,\dots,j_s\in\{1,\dots,s,q-(r-s)+1,\dots,q\}$ with $ j_1<j_2<\dots<j_s $ and $\{j_1,\dots,j_s\}\neq \{1,\dots,s\},$ such that
\[
   \alpha_1\cdot\alpha_2\cdots\alpha_s= \alpha_{j_1}\cdot\alpha_{j_2}\cdots\alpha_{j_s},
\]
which implies
\begin{equation}        \label{equ:sum(li)=sum(lj)_inproofof_1stCombiLemrefine}
   l_1+l_2+\dots+l_s=l_{j_1}+l_{j_2}+\dots+l_{j_s}.
\end{equation}
Since $ j_s\geq q-(r-s)+1,$ we know that
\[
   l_{j_s}\geq l_{q-(r-s)+1}>l_s.
\]
We also see that $ l_{j_1}\geq l_1 $, $ l_{j_2}\geq l_2 $, $\dots$, $ l_{j_{s-1}}\geq l_{s-1}.$ Thus we have
\[
   l_{j_1}+\dots+l_{j_{s-1}}+l_{j_s}> l_1+\dots+l_{s-1}+l_s,
\]
which contradicts the equality \eqref{equ:sum(li)=sum(lj)_inproofof_1stCombiLemrefine}.
Hence $ l_s=l_{q-(r-s)+1}.$ Then we get
\[
   l_s=l_{s+1}=\dots=l_{q-(r-s)+1},
\]
which implies
\[
   \alpha_s=\alpha_{s+1}=\dots=\alpha_{q-(r-s)+1}.
\]

Similarly, by considering the following $ r $ elements:
\[
   \alpha_1,\alpha_2,\dots,\alpha_{r-s}, \alpha_{q-s+1},\alpha_{q-s+2},\dots,\alpha_q,
\]
we conclude that
\[
   l_{r-s}=l_{r-s+1}=\dots=l_{q-s+1}.
\]

Now we assume that $ r\leq 2s.$

If $ q\geq 2s-1,$ which means $ q-s+1\geq s,$ then
\[
   l_{r-s}=l_{r-s+1}=\dots=l_{q-(r-s)+1}
\]
which implies
\[
   \alpha_{r-s}=\alpha_{r-s+1}=\dots=\alpha_{q-(r-s)+1}.
\]
If $ q<2s-1,$ which means $ q-s+1<s,$ then
\[
   \alpha_{r-s}=\alpha_{r-s+1}=\dots=\alpha_{q-s+1} \quad\mbox{and}\quad \alpha_s=\alpha_{s+1}=\dots=\alpha_{q-(r-s)+1}.
\]
This completes the proof of Lemma \ref{lem:1stCombiLem_refinement}.

Assume further that $ q\geq 2s-1 $ and $ r-s\geq 2.$ We now know that
\[
   l_1\leq \dots\leq l_{r-s-1}\leq l_{r-s}=l_{r-s+1}=\dots=l_{q-(r-s)+1}\leq l_{q-(r-s)+2}\leq \dots\leq l_q.
\]

If $ l_{r-s-1}=l_{r-s}$ or $ l_{q-(r-s)+1}=l_{q-(r-s)+2},$ which implies
\[
   \alpha_{r-s-1}=\alpha_{r-s} \quad\mbox{or}\quad \alpha_{q-(r-s)+1}=\alpha_{q-(r-s)+2},
\]
then the case ($\alpha$) of Lemma \ref{lem:2ndCombiLem_refinement} occurs.

So we assume $ l_{r-s-1}<l_{r-s}$ and $ l_{q-(r-s)+1}<l_{q-(r-s)+2}$ in the following. We now have
\[
   l_1\leq \dots\leq l_{r-s-1}<l_{r-s}=l_{r-s+1}=\dots=l_{q-(r-s)+1}<l_{q-(r-s)+2}\leq \dots\leq l_q.
\]

Choose the following $ r $ elements among $\alpha_i $'s:
\[
   \alpha_1,\dots,\alpha_{r-s},\alpha_{r-s+1},\dots,\alpha_s,\alpha_{q-(r-s)+1},\dots,\alpha_q.
\]
We note that $ s-(r-s)=2s-r.$

Consider the $ s $ elements $\alpha_{i_1},\alpha_{i_2},\dots,\alpha_{i_s},$ where
\[
   i_1=1,\dots, i_{r-s-2}=r-s-2, i_{r-s-1}=r-s,\dots, i_{s-1}=s, i_s=q-(r-s)+1.
\]
By assumption, there are indices $ j_1,\dots,j_s \in\{1,\dots,s,q-(r-s)+1,\dots,q\}$ with $ j_1<j_2<\dots<j_s $ and $\{j_1,\dots,j_s\}=:J\neq \{i_1,\dots,i_s\},$ such that
\[
   \alpha_{i_1}\cdot\alpha_{i_2}\cdots\alpha_{i_s}=\alpha_{j_1}\cdot\alpha_{j_2}\cdots\alpha_{j_s},
\]
which implies
\begin{equation}       \label{equ:sum(li)=sum(lj)_inproofof_2ndCombiLemrefine}
   l_{i_1}+l_{i_2}+\dots+l_{i_s}= l_{j_1}+l_{j_2}+\dots+l_{j_s}.
\end{equation}

First of all, we see that $ r-s-1\in J.$ For, if $ r-s-1\not\in J,$ then by $ l_{q-(r-s)+1}<l_{q-(r-s)+2}$ we see the right hand side of \eqref{equ:sum(li)=sum(lj)_inproofof_2ndCombiLemrefine} is strictly greater than the left hand side of \eqref{equ:sum(li)=sum(lj)_inproofof_2ndCombiLemrefine}, which gives a contradiction.

We assume $ r-s-1=j_{\kappa},$ where $ 1\leq \kappa\leq s.$ It is obvious that $ 1\leq \kappa\leq r-s-1.$ Next, we consider the following two cases.

\noindent $\bullet$ $\kappa<r-s-1.$ (This is possible only when $ r-s\geq 3,$ so we exclude this case when $ r-s=2.$)

Since $ l_{j_u}\geq l_{r-s}\geq l_{i_u}$ for any $\kappa+1\leq u\leq s,$ and $ l_{j_u}\geq l_{i_u}$ for any $ 1\leq u\leq \kappa,$ we conclude from \eqref{equ:sum(li)=sum(lj)_inproofof_2ndCombiLemrefine} that $ l_{j_u}=l_{i_u}$ for any $ 1\leq u\leq s.$ In particular, putting $ u=\kappa,$ we get $ l_{r-s-1}=l_{\kappa}.$ It follows that $ l_{r-s-1}=l_{r-s-2},$ which implies
\[
   \alpha_{r-s-2}=\alpha_{r-s-1}.
\]
This shows the case ($\beta$) of Lemma \ref{lem:2ndCombiLem_refinement} occurs.

\noindent $\bullet$ $\kappa=r-s-1.$

In this case, we necessarily have $ j_u=u $ for any $ 1\leq u\leq r-s-2.$ By
\[
   l_{j_{r-s-1}}(=l_{r-s-1})<l_{i_{r-s-1}}(=l_{r-s}),
\]
we conclude from \eqref{equ:sum(li)=sum(lj)_inproofof_2ndCombiLemrefine} that
\[
   \{j_{r-s},\dots,j_s\}\cap \{q-(r-s)+2,\dots,q\}\neq \emptyset.
\]
We assume that $\{j_{r-s},\dots,j_s\}\cap \{q-(r-s)+2,\dots,q\}=\{u_1,\dots,u_d\}.$ One sees easily that
\[
   1\leq d\leq \min\{2s-r+1,r-s-1\}.
\]
By cancelling some terms in the both sides of \eqref{equ:sum(li)=sum(lj)_inproofof_2ndCombiLemrefine} suitably, we get
\[
   l_{r-s}+d\cdot l_{r-s}=l_{r-s-1}+l_{u_1}+\dots+l_{u_d}.
\]
Since the $ l_i $ have the property \eqref{equ:sum(li)=sum(lj)_iff_prod(alphai)=prod(alphaj)}, we obtain
\[
   \alpha_{r-s-1}\cdot\alpha_{u_1}\cdots\alpha_{u_d}=\alpha_{r-s}^{d+1},
\]
which shows the case ($\gamma$) of Lemma \ref{lem:2ndCombiLem_refinement} occurs.

The proof of Lemma \ref{lem:2ndCombiLem_refinement} is now completed.

\stepcounter{section}
\section*{Appendix \Alph{section}\hspace{0.73em} The algebraic non-degeneracy of $g$ in Example \ref{example:N>n=1_q=N+2_g_algebrai._nondege.}}
\addcontentsline{toc}{section}{Appendix \Alph{section}\,\, The algebraic non-degeneracy of $g$ in Example \ref{example:N>n=1_q=N+2_g_algebrai._nondege.}}

We shall freely use the standard notations and results from Nevanlinna theory.  For details of Nevanlinna theory, the reader may refer to books \cite{Ru01book} and \cite{NoguchiWinkelmann14book}.

We need the following lemma.
\begin{lem}        \label{lem:T(r,anfn+...+a0)<=nT(r,f)+O()}
Let $ f $ and $ a_0, a_1,\dots, a_n $ be meromorphic functions on $\mathbb{C}.$ Then
 \[
    T\big(r,a_nf^n+a_{n-1}f^{n-1}+\dots+a_1f+a_0\big)\leq n T(r,f)+O\Big(\sum_{i=0}^n T(r,a_i)\Big)+O(1).
 \]
\end{lem}
\begin{proof}
The proof is easy by using the induction on $ n.$ The details are omitted here.
\end{proof}

We have the following result.
\begin{prop}       \label{prop:(1,ez,...,ezn)_is_algebrai._nondege.}
For any integer $ n\geq 1,$ the holomorphic map $ F_n $ of $\mathbb{C}$ into $\mathbb{P}^n $ which is defined as
 \[
    F_n(z)=[1:e^z:e^{z^2}:\dots: e^{z^n}],
 \]
is algebraically non-degenerate.
\end{prop}
\begin{proof}
With the help of Lemma \ref{lem:T(r,anfn+...+a0)<=nT(r,f)+O()}, this proposition can be easily proved by the induction on $ n.$ The details are also omitted here.
\end{proof}

Now we are going to show that the holomorphic map $ g $ in Example \ref{example:N>n=1_q=N+2_g_algebrai._nondege.} is algebraically non-degenerate. Let notations be as in Example \ref{example:N>n=1_q=N+2_g_algebrai._nondege.}.

Put
\[
   \hat{g}_j:= g_j,\quad 0\leq j\leq N-1,\quad\mbox{and}\quad \hat{g}_N:=(1-a_N)g_0+g_1+\dots+g_N.
\]
By $ g_0+\dots+g_N=-f_1,$ we get $\hat{g}_N=a_Nf_0-f_1.$ From the expressions of $ f_0 $ and $ f_1,$ we get the following expressions of the functions $ F\hat{g}_j\, (0\leq j\leq N)$:
\begin{align*}
    & F\hat{g}_0=-\big(1+e^{G_1}+\dots+e^{G_N}\big),
 \\ & F\hat{g}_j=e^{G_j}\Big[(1-a_j)+\sum_{i=1,i\neq j}^{N-1}(a_i-a_j)e^{G_i}+(a_N-a_j)e^{G_N}\Big], \quad 1\leq j\leq N-1,
 \\ & F\hat{g}_N =(a_N-1)+\sum_{i=1}^{N-1}(a_N-a_i)e^{G_i}.
\end{align*}

Assume $ g $ is not algebraically non-degenerate. Then we see easily that there is a homogeneous polynomial $ Q(X_0,\dots,X_N)$ of degree $ d\geq 1 $ such that
\[
   Q(F\hat{g}_0,F\hat{g}_1,\dots,F\hat{g}_N)\equiv 0.
\]
Noticing that $ F\hat{g}_N\not\equiv 0,$ we conclude that the above equation can be written in the following form:
\begin{align}        \label{equ:P(1,eG1,...,eG(N-1))eGN^(d-k)(FgN)^k+...equiv0}
   P(1,e^{G_1}&,\dots,e^{G_{N-1}})\cdot(e^{G_N})^{d-k}\cdot(F\hat{g}_N)^k     \notag
 \\ &+P_{d-k-1}(e^{G_1},\dots,e^{G_{N-1}})\cdot(e^{G_N})^{d-k-1}+\cdots       \notag
 \\ &\hspace{2em}+P_1(e^{G_1},\dots,e^{G_{N-1}})\cdot e^{G_N}+P_0(e^{G_1},\dots,e^{G_{N-1}})\equiv 0,
\end{align}
where $ k $ is an integer with $ 0\leq k\leq d-1,$ and $ P(X_0,X_1,\dots,X_{N-1})$ is a nonzero homogeneous polynomial of degree $(d-k),$ and $ P_{d-k-1},\dots, P_0 $ are polynomials.

By Proposition \ref{prop:(1,ez,...,ezn)_is_algebrai._nondege.}, we know $ P(1,e^{G_1},\dots,e^{G_{N-1}})\not\equiv 0.$ We note that $ T(r,e^{G_i})=o\big(T(r,e^{G_N})\big)$ for any $ 1\leq i\leq N-1.$ So applying Lemma \ref{lem:T(r,anfn+...+a0)<=nT(r,f)+O()}, we get from \eqref{equ:P(1,eG1,...,eG(N-1))eGN^(d-k)(FgN)^k+...equiv0} that
\[
   (d-k)T(r,e^{G_N})+O(1)\leq (d-k-1)T(r,e^{G_N})+o\big(T(r,e^{G_N})\big).
\]
This implies
\[
   d-k\leq d-k-1,
\]
which is a contradiction.

Thus $ g $ is algebraically non-degenerate.

\addcontentsline{toc}{section}{References}
\bibliographystyle{plain}
\bibliography{zkRef}

\begin{thebibliography}{10}

\bibitem{Aihara00}
Yoshihiro Aihara.
\newblock Unicity theorems for meromorphic mappings with deficiencies.
\newblock {\em Complex Variables}, 42(3):259--268, 2000.

\bibitem{Borel1897}
Emile Borel.
\newblock Sur les z\'eros des fonctions enti\`eres.
\newblock {\em Acta Math.}, 20:357--396, 1897.

\bibitem{ChenYan09}
Zhihua Chen and Qiming Yan.
\newblock {Uniqueness theorem of meromorphic mappings into
  $\mathbb{P}^N(\mathbb{C})$ sharing $2N+3$ hyperplanes regardless of
  multiplicities}.
\newblock {\em Internat. J. Math.}, 20(6):717--726, 2009.

\bibitem{CorvajaNoguchi12}
Pietro Corvaja and Junjiro Noguchi.
\newblock A new unicity theorem and {Erd\"os'} problem for polarized
  semi-abelian varieties.
\newblock {\em Math. Ann.}, 353(2):439--464, 2012.

\bibitem{Drouilhet81}
S.~J. Drouilhet.
\newblock A unicity theorem for meromorphic mappings between algebraic
  varieties.
\newblock {\em Trans. Amer. Math. Soc.}, 265(2):349--358, 1981.

\bibitem{Fujimoto75}
Hirotaka Fujimoto.
\newblock The uniqueness problem of meromorphic maps into the complex
  projective space.
\newblock {\em Nagoya Math. J.}, 58:1--23, 1975.

\bibitem{Fujimoto76}
Hirotaka Fujimoto.
\newblock A uniqueness theorem of algebraically non-degenerate meromorphic maps
  into {$P^N(\mathbf{C})$}.
\newblock {\em Nagoya Math. J.}, 64:117--147, 1976.

\bibitem{Fujimoto78}
Hirotaka Fujimoto.
\newblock Remarks to the uniqueness problem of meromorphic maps into
  {$P^N(\mathbf{C})$, I}.
\newblock {\em Nagoya Math. J.}, 71:13--24, 1978.

\bibitem{Fujimoto99}
Hirotaka Fujimoto.
\newblock Uniqueness problem with truncated multiplicities in value
  distribution theory, {II}.
\newblock {\em Nagoya Math. J.}, 155:161--188, 1999.

\bibitem{Nevanlinna1926}
Rolf Nevanlinna.
\newblock {Einige Eindeutigkeitss\"atze in der Theorie der meromorphen
  Funktionen}.
\newblock {\em Acta Math.}, 48(3-4):367--391, 1926.

\bibitem{NoguchiWinkelmann14book}
Junjiro Noguchi and J\"org Winkelmann.
\newblock {\em {Nevanlinna Theory in Several Complex Variables and Diophantine
  Approximation}}.
\newblock Springer, Tokyo, 2014.

\bibitem{Polya1921}
George P\'olya.
\newblock {Bestimmung einer ganzen Funktion endlichen Geschlechts durch
  viererlei Stellen}.
\newblock {\em Mat. Tidsskrift B}, pages 16--21, 1921.
\newblock https://zbmath.org/?q=an\%3A48.0354.03.

\bibitem{Quang11}
Si~Duc Quang.
\newblock Unicity of meromorphic mappings sharing few hyperplanes.
\newblock {\em Ann. Polon. Math.}, 102(3):255--270, 2011.

\bibitem{Ru01book}
Min Ru.
\newblock {\em {Nevanlinna Theory and Its Relation to Diophantine
  Approximation}}.
\newblock World Scientific Publishing Co., Inc., River Edge, NJ, 2001.

\bibitem{Smiley83}
Leonard~M. Smiley.
\newblock Geometric conditions for unicity of holomorphic curves.
\newblock {\em Contemp. Math.}, 25:149--154, 1983.

\bibitem{Yamanoi04_2}
Katsutoshi Yamanoi.
\newblock Holomorphic curves in {Abelian} varieties and intersections with
  higher codimensional subvarieties.
\newblock {\em Forum Math.}, 16(5):749--788, 2004.

\bibitem{zk23_ANote}
Kai Zhou.
\newblock A note on {Fujimoto's} uniqueness theorem with $(2n+3)$ hyperplanes.
\newblock 2023.
\newblock {\tt arXiv:2307.16595 [math.CV]}.

\end{thebibliography}

\vspace{2em}
\noindent School of Mathematical Sciences, Tongji University, Shanghai 200092, People's Republic of China   \par
\noindent {\it E-mail address}: \texttt{zhoukai@tongji.edu.cn}

\end{document}